\documentclass[11pt]{amsart}

\usepackage[pdfauthor={Ziyang GAO}, 
        pdftitle={Ax-Lindemann}%
      dvips,colorlinks=true]{hyperref}
\hypersetup{
  bookmarksnumbered=true,
  linkcolor=black,
  citecolor=black,
  pagecolor=black, 
  urlcolor=black,  
}

\usepackage[utf8]{inputenc}

\usepackage[french,english]{babel} 

\usepackage{xtab}
\usepackage{amsmath, amssymb, cancel,verbatim}
\usepackage[all]{xy}

\usepackage{mathtools,booktabs}
\usepackage{amscd}
\usepackage{bbm, stmaryrd}
\usepackage{yfonts}
\usepackage{color}

\usepackage{epstopdf} 
\usepackage{booktabs}

\newtheorem{teo}{Theorem}[section]
\newtheorem{thm}[teo]{Theorem}
\newtheorem{prop}[teo]{Proposition}
\newtheorem{lemma}[teo]{Lemma}
\newtheorem{cor}[teo]{Corollary}
\newtheorem{conj}[teo]{Conjecture}

\newtheorem{defn}[teo]{Definition}

\newtheorem{rmk}[teo]{Remark}

\newtheorem*{claim}{Claim}
\newtheorem{def-prop}[teo]{Definition-Proposition}
\newtheorem{notation}[teo]{Notation}

\numberwithin{equation}{section}

\usepackage[small,nohug,heads=LaTeX]{diagrams}
\diagramstyle[labelstyle=\scriptstyle]

  \newcommand{\B}{\mathbb{B}}
  \newcommand{\C}{\mathbb{C}}

  \newcommand{\G}{\mathbb{G}}
  \renewcommand{\H}{\mathbb{H}}  
  
  \newcommand{\N}{\mathbb{N}}
  
  \newcommand{\Q}{\mathbb{Q}}
  \newcommand{\R}{\mathbb{R}}  
  \renewcommand{\S}{\mathbb{S}}
  
  \newcommand{\Z}{\mathbb{Z}}

  \renewcommand{\epsilon}{\varepsilon}
  \renewcommand{\cong}{\simeq}
  \renewcommand{\bar}{\overline}
  \renewcommand{\tilde}{\widetilde}
  \renewcommand{\hat}{\widehat}

  \providecommand{\frac}[1]{\operatorname{Frac}(#1)}

  \renewcommand{\hom}{\operatorname{Hom}}

  \newcommand{\rank}{\operatorname{rank}}
  
  \newcommand{\gal}{\operatorname{Gal}}

  \newcommand{\MT}{\operatorname{MT}}
  \newcommand{\GL}{\operatorname{GL}}

  \renewcommand{\lim}{\operatorname{lim}}
  \renewcommand{\deg}{\operatorname{deg}}

  \newcommand{\aut}{\operatorname{Aut}}
  \newcommand{\End}{\operatorname{End}}
  \newcommand{\Int}{\operatorname{int}}

  \newcommand{\sm}{\mathrm{sm}}

  \newcommand{\Zar}{\mathrm{Zar}}
  \newcommand{\GSp}{\mathrm{GSp}}

  \newcommand{\res}{\mathrm{Res}}

  \newcommand{\unif}{\mathrm{unif}}

\newcommand{\cA}{\mathcal{A}}
\newcommand{\cB}{\mathcal{B}}
\newcommand{\cC}{\mathcal{C}}

\newcommand{\cF}{\mathcal{F}}

\newcommand{\cM}{\mathcal{M}}

\newcommand{\cR}{\mathcal{R}}

\newcommand{\cX}{\mathcal{X}}
\newcommand{\cY}{\mathcal{Y}}

\newcommand{\upperarrow}[1]{
\setlength{\unitlength}{0.03\DiagramCellWidth}
\begin{picture}(0,0)(0,0)
\linethickness{1pt} 
\bezier{15}(-46,9)(-24,23)(-2,9)
\put(-24,19){\makebox(0,0)[b]{$\scriptstyle {#1}$}}
\put(-46.6,8.7){\vector(-2,-1){0}}
\end{picture}
}
\newcommand{\upperarrowtwo}[1]{
\setlength{\unitlength}{0.03\DiagramCellWidth}
\begin{picture}(0,0)(0,0)
\linethickness{1pt} 
\bezier{15}(-40,9)(-21,23)(-2,9)
\put(-21,19){\makebox(0,0)[b]{$\scriptstyle {#1}$}}
\put(-40.6,8.7){\vector(-2,-1){0}}
\end{picture}
}

\newcommand\supervisor[1]{\def\@supervisor{#1}}

\newcounter{elno}

\renewcommand{\cong}{\simeq}
\begin{document}
\title[The Andr\'{e}-Pink-Zannier conjecture]{A special point problem of Andr\'{e}-Pink-Zannier in the universal family of abelian varieties}
\author{Ziyang Gao}
\address{Institut des Hautes \'{E}tudes Scientifiques, 
Le Bois-Marie 35, route de Chartres, 
91440 Bures-sur-Yvette,
France}
\email{ziyang.gao@math.u-psud.fr}
\maketitle

\begin{abstract} The Andr\'{e}-Pink-Zannier conjecture predicts that a subvariety of a mixed Shimura variety is weakly special if its intersection with the generalized Hecke orbit of a given point is Zariski dense. It is part of the Zilber-Pink conjecture. In this paper we focus on the universal family of principally polarized abelian varieties. We explain the moduli interpretation of the Andr\'{e}-Pink-Zannier conjecture in this case and prove several different cases for this conjecture: its overlap with the Andr\'{e}-Oort conjecture; when the subvariety is contained in an abelian scheme over a curve and the point is a torsion point on its fiber; when the subvariety is a curve.
\end{abstract}

%\selectlanguage{french}
%\begin{abstract}[\textbf{Un probl\`{e}me de type \og points-sp\'{e}ciaux\fg{}~d'Andr\'{e}-Pink-Zannier pour la famille universelle de vari\'{e}t\'{e}s ab\'{e}liennes}] La conjecture d'Andr\'{e}-Pink-Zannier conjecture, selon laquelle une sous-vari\'{e}t\'{e} d'une vari\'{e}t\'{e} de Shimura mixte est faiblement sp\'{e}ciale si son intersection avec l'orbite de Hecke g\'{e}n\'{e}ralis\'{e}e d'un point fix\'{e} est Zariski dense, fait partie de la conjecture de Zilber-Pink. Dans cet article nous \'{e}tudions la famille universelle de vari\'{e}t\'{e}s ab\'{e}liennes principalement polaris\'{e}es. Nous expliquons l'interpretation modulaire de la conjecture d'Andr\'{e}-Pink-Zannier dans ce cas et nous d\'{e}monstrons plusieurs cas pour cette conjecture~: son intersection avec la conjecture d'Andr\'{e}-Oort; quand la sous-vari\'{e}t\'{e} est contenue dans un sch\'{e}ma ab\'{e}lien au-dessus d'une courbe et que le point est un point de torsion sur sa fibre; quand la sous-vari\'{e}t\'{e} est une courbe est que le point est un point alg\'{e}brique.
%\end{abstract}

%\selectlanguage{english}
\tableofcontents

\section{Introduction}

Consider $[\pi]\colon\mathfrak{A}_g(N)\rightarrow\cA_g(N)$, the universal family of principally polarized abelian varieties of dimension $g$ with level-$N$-structure over a fine moduli space. For simplicity we drop the ``$(N)$" in the notation. The variety $\mathfrak{A}_g$ is an example of a mixed Shimura variety which is not pure. For general theory of mixed Shimura varieties, we refer to \cite{PinkThesis}. An interesting Diophantine problem related to mixed Shimura varieties is the Zilber-Pink conjecture, which concerns unlikely intersections in mixed Shimura varieties. In order to study this conjecture, Pink defined in \cite[Definition 4.1]{PinkA-Combination-o} \textbf{weakly special subvarieties} of mixed Shimura varieties. In $\mathsection$\ref{SectionWeaklySpecialSubvariety}, we shall discuss weakly special subvarieties of $\mathfrak{A}_g$. In particular we dispose of the following geometric description for weakly special subvarieties of $\mathfrak{A}_g$: let $Y$ be any irreducible subvariety of $\mathfrak{A}_g$, it is then a subvariety of $[\pi]^{-1}([\pi]Y)$ with the latter being an abelian scheme over $[\pi]Y$, whose isotrivial part we denote by $\cC$. Then we have (for proof see $\mathsection$\ref{GeometricInterpretationWeaklySpecial})

\begin{prop}\label{DescriptionWeaklySpecial}
An irreducible subvariety $Y$ of $\mathfrak{A}_g$ is weakly special iff the following holdspolarized:
\begin{enumerate}
\item $[\pi]Y$ is a totally geodesic subvariety of $\cA_g$;
\item $Y$ is the translate of an abelian subscheme of $[\pi]^{-1}([\pi]Y)$ (over $[\pi]Y$) by a torsion section and then by a section of $\cC\rightarrow[\pi]Y$.
\end{enumerate}
Moreover, this holds for any connected Shimura variety of Kuga type $S$ (i.e. mixed Shimura varieties with trivial weight $-2$ part), in which case the ``$\cA_g$" in (1) should be replaced by the pure part of $S$. See the forthcoming dissertation \cite[Section 2.2]{GaoThesis}.
\end{prop}
Let us define constant sections of $\cC\rightarrow[\pi]Y$. By definition of isotriviality, there exists a finite cover $B^\prime\rightarrow B$ such that $\cC\times_{[\pi]Y}B^\prime\cong\cC_{b_0}\times B^\prime$ for any $b_0\in[\pi]Y$. A \textbf{constant section of $\cC\rightarrow[\pi]Y$} is then defined to be the image of the graph of a constant morphism $B^\prime\rightarrow\cC_{b_0}$ in $\cC\times_{[\pi]Y}B^\prime$ under the projection $\cC\times_{[\pi]Y}B^\prime\rightarrow\cC$.

A very important case of the Zilber-Pink conjecture is the Andr\'{e}-Oort conjecture, which for $\mathfrak{A}_g$ is equivalent to the following statement: if a subvariety $Y$ of $\mathfrak{A}_g$ contains a Zariski dense subset of special points (i.e. points of $\mathfrak{A}_g$ corresponding to torsion points of CM abelian varieties), then $Y$ is a weakly special subvariety of $\mathfrak{A}_g$. By previous work of Pila-Tsimerman \cite{PilaAxLindemannAg} and Gao \cite{GaoTowards-the-And}, the only obstacle to prove the Andr\'{e}-Oort conjecture for $\mathfrak{A}_g$ (or more generally, for any mixed Shimura variety of abelian type) is the lower bound for the Galois-orbits of special points.

The goal of this article is to study another important case of the Zilber-Pink conjecture, which we call the Andr\'{e}-Pink-Zannier conjecture:
\begin{conj}\label{AndrePinkConjecture}
Let $Y$ be a subvariety of $\mathfrak{A}_g$. Let $s\in\mathfrak{A}_g$ and $\Sigma$ be the generalized Hecke orbit of $s$. If $\bar{Y\cap\Sigma}^{\Zar}=Y$, then $Y$ is weakly special.
\end{conj}

Several cases of this conjecture had been studied by Andr\'{e} before its final form was made by Pink \cite[Conjecture 1.6]{PinkA-Combination-o}. It is also closely related to a problem (Conjecture~\ref{ConjectureZannier}) proposed by Zannier. Pink has also proved \cite[Theorem 5.4]{PinkA-Combination-o} that Conjecture~\ref{AndrePinkConjecture} implies the Mordell-Lang conjecture.

Conjecture~\ref{AndrePinkConjecture} for $\cA_g$, the pure part of $\mathfrak{A}_g$, has been intensively studied by Orr in \cite{OrrFamilies-of-abe, MartinThesis}, generalizing the previous work of Habegger-Pila \cite[Theorem 3]{HabeggerSome-unlikely-i} in the Pila-Zannier method. This paper is based on the work of Orr \cite{OrrFamilies-of-abe, MartinThesis} and the author's previous work on the mixed Andr\'{e}-Oort conjecture \cite{GaoTowards-the-And}.

The set $\Sigma$ has good moduli interpretation: by Corollary \ref{GeneralizedHeckeIsogenyStep1},
\begin{equation}\label{ModuliInterpretationOfSigma}
\begin{array}{ll}
\Sigma&=\text{division points of the polarized isogeny orbit of }s \\
&=\{t\in\mathfrak{A}_g|~\exists n\in\N\text{ and a polarized isogeny }f\colon(\mathfrak{A}_{g,[\pi]s},\lambda_{[\pi]s})\rightarrow(\mathfrak{A}_{g,[\pi]t},\lambda_{[\pi]t})\text{ such that }nt=f(s)\}.
\end{array}
\end{equation}
There are authors who consider isogenies instead of polarized isogenies. However this does not essentially improve the result because of Zarhin's trick (see \cite[Proposition~4.4]{MartinThesis}): for any isogeny $f\colon A\rightarrow A^\prime$ between polarized abelian varieties, there exists $u\in\End(A^4)$ such that $f^4\circ u\colon A^4\rightarrow A^{\prime4}$ is a polarized isogeny. See $\mathsection$\ref{SectionVariantsOfTheMainConjecture} for more details.

Although Conjecture~\ref{AndrePinkConjecture} and the Andr\'{e}-Oort conjecture do not imply each other, they do have some overlap. The overlap of these two conjectures is the same statement of Conjecture~\ref{AndrePinkConjecture} with $\Sigma$ replaced by \textit{the set of points of $\mathfrak{A}_g$ corresponding to torsion points of CM abelian varieties admitting a polarized isogeny to a given principally polarized CM abelian variety}. A main result of this paper is to prove this overlap, partially generalizing existing result of Edixhoven-Yafaev \cite{YafaevSous-varietes-d, EdixhovenSubvarieties-of} and Klingler-Ullmo-Yafaev \cite{KlinglerThe-Andre-Oort-, UllmoGalois-orbits-a} for pure Shimura varieties  (see Theorem~\ref{MainTheoremZannier}.(2)).

We shall divide Conjecture~\ref{AndrePinkConjecture} into two cases: when $s$ is a torsion point of $\mathfrak{A}_{g,[\pi]s}$ and when $s$ is not a torsion point of $\mathfrak{A}_{g,[\pi]s}$. The diophantine estimates for both cases are not quite the same.

\subsection{The torsion case}
When $s$ is a torsion point of $\mathfrak{A}_{g,[\pi]s}$, this conjecture is related to a special-point problem proposed by Zannier. We define the following ``special topology" proposed by Zannier:
\begin{defn}
Fix a point $a\in\cA_g$. Then $a$ corresponds to a principally polarized abelian variety $(A_a,\lambda_a)$ of dimension $g$.
\begin{enumerate}
\item
We say that a point $t\in\mathfrak{A}_g$ is \textbf{$A_a$-special} (or \textbf{$a$-special}) if there exists an isogeny $A_a\rightarrow\mathfrak{A}_{g,[\pi]t}$ and that $t$ is a torsion point on the abelian variety $\mathfrak{A}_{g,[\pi]t}$. We shall denote by $\Sigma^\prime_a$ (or $\Sigma^\prime$ when there is no confusion) the set of $a$-special points.
\item
We say that a point $t\in\mathfrak{A}_g$ is \textbf{$(A_a,\lambda_a)$- special} if there exists a polarized isogeny $(A_a,\lambda_a)\rightarrow(\mathfrak{A}_{g,[\pi]t},\lambda_{[\pi]t})$ and that $t$ is a torsion point on the abelian variety $\mathfrak{A}_{g,[\pi]t}$. We shall denote by $\Sigma_a$ (or $\Sigma$ when there is no confusion) the set of $a$-strongly special points.
\item
We say that a subvariety $Z$ of $\mathfrak{A}_g$ is \textbf{$a$-special} if $Z$ contains an $a$-special point, $[\pi]Z$ is a totally geodesic subvariety of $\cA_g$ and $Z$ is an irreducible component of a subgroup of $[\pi]^{-1}([\pi]Z)$.
\end{enumerate}
\end{defn}

In view of Proposition~\ref{DescriptionWeaklySpecial}, every $a$-(strongly) special subvariety is weakly special. The following conjecture is proposed by Zannier.

\begin{conj}\label{ConjectureZannier}
Let $Y$ be a subvariety of $\mathfrak{A}_g$ and let $a\in\cA_g$. If $\bar{Y\cap\Sigma^\prime_a}^{\Zar}=Y$, then $Y$ is $a$-special.
\end{conj}

By \eqref{ModuliInterpretationOfSigma}, Conjecture~\ref{AndrePinkConjecture} when $s$ is a torsion point of $\mathfrak{A}_{g,[\pi]s}$ is equivalently to a weaker version of Conjecture~\ref{ConjectureZannier}, i.e. replace $\Sigma^\prime_a$ by $\Sigma_a$ in Conjecture~\ref{ConjectureZannier}. However by \cite[Proposition 4.4]{MartinThesis}, Conjecture~\ref{AndrePinkConjecture} for $\mathfrak{A}_{4g}$ also implies Conjecture~\ref{ConjectureZannier} for $\mathfrak{A}_g$. Our first main result is:

\begin{thm}\label{MainTheoremZannier}
Conjecture~\ref{ConjectureZannier} holds if one of the following conditions holds:
\begin{enumerate}
\item either $\dim([\pi](Y))\leqslant 1$;
\item or the point $a$ is a special point of $\cA_g$ (which is the overlap of Conjecture~\ref{AndrePinkConjecture} and the Andr\'{e}-Oort conjecture for $\mathfrak{A}_g$).
\end{enumerate}
\end{thm}

The proof of this theorem will be presented in $\mathsection$\ref{SectionDiophantineEstimateForTheTorsionCase} and $\mathsection$\ref{SectionEndOfTheProofForTheTorsionCase}.
Remark that by Corollary \ref{GeneralizedHeckeIsogeny}, the case where $\dim([\pi]Y)=0$ (i.e. $[\pi](Y)$ is a point) is nothing but the Manin-Mumford conjecture, which is proved by many people (the first proof was given by Raynaud). On the other hand, with a similar proof, Theorem~\ref{MainTheoremZannier}.(2) holds for more general cases (more details will be given in the forthcoming dissertation \cite[Theorem 14.2]{GaoThesis}). In this paper we only present the proof for the case $\mathfrak{A}_g$.

\subsection{The non-torsion case}
The situation becomes more complicated when $s$ is not a torsion point of $\mathfrak{A}_{g,[\pi]s}$. In this case we prove (in $\mathsection$\ref{SectionProofOfTheNonTorsionCase}):
\begin{thm}\label{MainTheoremNonTorsion}
Conjecture~\ref{AndrePinkConjecture} holds if $Y$ is a curve.
\end{thm}

\subsection*{Structure of the paper} In $\mathsection$\ref{SectionUniversalFamily} we define the universal family of abelian varieties in the language of mixed Shimura varieties of Pink \cite{PinkThesis}. In $\mathsection$\ref{SectionWeaklySpecialSubvariety} we discuss weakly special subvarieties of $\mathfrak{A}_g$. In particular we prove Proposition~\ref{DescriptionWeaklySpecial} and recall the Ax-Lindemann theorem in this section. Then we shall lay the base of the study for Conjecture~\ref{AndrePinkConjecture} in $\mathsection$\ref{SectionGeneralizedHeckeOrbit}, where matrix expressions of polarized isogenies are given and generalized Hecke orbits are computed. After these preliminaries, we will start proving Theorem~\ref{MainTheoremZannier} and Theorem~\ref{MainTheoremNonTorsion}. The proof of Theorem~\ref{MainTheoremZannier} will be executed in $\mathsection$\ref{SectionDiophantineEstimateForTheTorsionCase} and $\mathsection$\ref{SectionEndOfTheProofForTheTorsionCase}, with the former section devoted to the Diophantine estimate and the latter section devoted to the rest of the proof. In $\mathsection$\ref{SectionProofOfTheNonTorsionCase} the proof for Theorem~\ref{MainTheoremNonTorsion} will be presented. In the last section $\mathsection$\ref{SectionVariantsOfTheMainConjecture}, we discuss the following situation: replace the subset $\Sigma$ (which is \eqref{ModuliInterpretationOfSigma}) in Conjecture~\ref{AndrePinkConjecture} by the isogeny orbit of a finitely generated subgroup of one fiber. We will prove that although this change a priori seems to generalize Conjecture~\ref{AndrePinkConjecture}, it can in fact be implied by Conjecture~\ref{AndrePinkConjecture}. For more details see Corollary \ref{AndrePinkFinitelyGeneratedSubgroupIsogenyOrbit}.

\subsection*{Achknowledgements} I am grateful to my supervisor Emmanuel Ullmo for regular discussions during the preparation of this paper. I would like to thank Martin Orr a lot for answering my questions related to his previous work \cite{MartinThesis, OrrFamilies-of-abe}. I would like to thank Nicolas Ratazzi for pointing out the paper of David \cite{DavidMinorations-de-} to me. This article was merged into my PhD thesis \cite{GaoThesis}. I would like to thank Yves Andr\'{e}, Bas Edixhoven and Bruno Klingler for their careful reading and suggestions to improve the presentation of the paper. I would also like to thank Daniel Bertrand, Marc Hindry and David Holmes for relevant discussion. Finally I would like to thank the referee for their reading of the manuscript and their helpful comments.

\section{Universal family of abelian varieties}\label{SectionUniversalFamily}
Let $\S:=\res_{\C/\R}\G_{m,\C}$. 
Let $g\in\N_{>0}$. Let $V_{2g}$ be a $\Q$-vector space of dimension $2g$ and let
\begin{equation}\label{AlternatingForm}
\Psi\colon V_{2g}\times V_{2g}\rightarrow U_{2g}:=\G_{a,\Q}
\end{equation}
be a non-degenerate alternating form. Define
\[
\GSp_{2g}:=\{g\in\GL(V_{2g})|\Psi(gv,gv^\prime)=\nu(g)\Psi(v,v^\prime)\text{ for some }\nu(g)\in\G_m\},
\]
and $\H_g^+$ the set of all homomorphisms
\[
\S\rightarrow\GSp_{2g,\R}
\]
which induce a pure Hodge structure of type $\{(-1,0),(0,-1)\}$ on $V_{2g}$ and for which $\Psi$ defines a polarization. The action of $\GSp_{2g}(\R)^+$ on $\H_g^+$ is given by the conjugation, i.e. for any $h\in\GSp_{2g}(\R)^+$ and any $x\in\H_g^+$, $h\cdot x$ is the morphism
\begin{align*}
h\cdot x\colon\S &\rightarrow\GSp_{2g,\R} \\
y &\mapsto hx(y)h^{-1}
\end{align*}
It is well known that $\H^+_g$ can be identified with the Siegel upper half space (of genus $g$)
\[
\{Z=X+\sqrt{-1}Y\in M_{g\times g}(\C)|~Z=Z^t,~Y>0\}
\]
and the action of $\GSp_{2g}(\R)^+$ on $\H_g^+$ is given by
\[
\left(\begin{array}{cc}
A & B \\
C & D
\end{array}\right)Z:=(AZ+B)(CZ+D)^{-1}.
\]

The action of $\GSp_{2g}$ on $V_{2g}$ induces a Hodge structure of type $\{(-1,0),(0,-1)\}$ on $V_{2g}$. Let
\[
\cX_{2g,a}^+:=V_{2g}(\R)\rtimes\H_g^+\subset\hom(\S,V_{2g,\R}\rtimes\GSp_{2g,\R})
\]
denote the conjugacy class under $(V_{2g}\rtimes\GSp_{2g})(\R)^+$ generated by $\H_g^+$ (recall that every point of $\H_g^+$ gives rise to a homomorphism $\S\rightarrow\GSp_{2g,\R}\subset V_{2g,\R}\rtimes\GSp_{2g,\R}$). The notion $V_{2g}(\R)\rtimes\H_g^+$ is justified by the natural bijection
\begin{equation}\label{UniformizingSpace}
V_{2g}(\R)\times\H_g^+\xrightarrow{\sim}V_{2g}(\R)\rtimes\H_g^+,~~(v^\prime,x)\mapsto\mathrm{int}(v^\prime)\circ x.
\end{equation}
Under this bijection the action of $(v,h)\in(V_{2g}\rtimes\GSp_{2g})(\R)^+$ is given by $(v,h)\cdot(v^\prime,x):=(v+hv^\prime,hx)$.

Denote by $(P_{2g,\mathrm{a}},\cX_{2g,a}^+):=(V_{2g}\rtimes\GSp_{2g},V_{2g}(\R)\times\H_g^+)$. This is a connected mixed Shimura datum (\cite[2.25]{PinkThesis}). There is a natural morphism
\[
\pi\colon(P_{2g,\mathrm{a}},\cX_{2g,a}^+)\rightarrow(\GSp_{2g},\H_g^+)
\]
induced by $P_{2g,\mathrm{a}}=V_{2g}\rtimes\GSp_{2g}\rightarrow\GSp_{2g}$.

Let $\Gamma_V(N):=NV(\Z)$ and $\Gamma_G(N):=\{h\in \GSp_{2g}(\Z)|h\equiv 1\pmod N\}$ for any integer $N\geqslant 3$. Define $\Gamma(N):=\Gamma_V(N)\rtimes \Gamma_G(N)$, then it is a neat subgroup of $P_{2g}(\Q)^+$. Define
\[
\mathfrak{A}_g(N):=\Gamma(N)\backslash\cX_{2g,a}^+
\]
and
\[
\cA_g(N):=\Gamma_G(N)\backslash\H_g^+.
\]
Then $\mathfrak{A}_g(N)$ is a connected mixed Shimura variety and $\cA_g(N)$ is a connected pure Shimura variety. The morphism $\pi$ induces a Shimura morphism
\begin{equation}\label{UniversalFamily}
[\pi]\colon\mathfrak{A}_g(N)\rightarrow\cA_g(N).
\end{equation}

\begin{thm}\label{UniversalFamilyTheorem}
\begin{enumerate}
\item The morphism \eqref{UniversalFamily} is the universal family of principally polarized abelian varieties of dimension $g$ over the fine moduli space $\cA_g(N)$.
\item Both $\mathfrak{A}_g(N)$ and $\cA_g(N)$ are both defined over $\bar{\Q}$.
\item Let $\cF:=[0,N)^{2g}\times\cF_G\subset V_{2g}(\R)\times\H_g^+\cong\cX^+_{2g,\mathrm{a}}$, where $\cF_G$ is a fundamental Siegel set for the action of $\Gamma_G(N)$ on $\H_g^+$. Then $\cF$ is a fundamental set for the action of $\Gamma(N)$ on $\cX^+_{2g,\mathrm{a}}$ such that $\unif|_{\cF}$ is definable in the o-minimal theory $\R_{an,\exp}$.
\end{enumerate}
\begin{proof} See \cite[10.5, 10.9, 10.10, 11.16]{PinkThesis} for (1) and (2). (3) is the main result of \cite{PeterzilDefinability-of} (see \cite[Remark 4.4]{GaoTowards-the-And}).
\end{proof}
\end{thm}

Let $N\geqslant 3$ be even. Pink has also constructed an ample $\G_m$-torsor over $\mathfrak{A}_g(N)$ in terms of mixed Shimura varieties in \cite{PinkThesis}. In our purpose we only need:

\begin{thm}\label{AmpleLineBundleOverTheUniversalFamily}
There exists a $\G_m$-torsor $\mathfrak{L}_g(N)\rightarrow\mathfrak{A}_g(N)$, which is totally symmetric and relatively ample with respect to $\mathfrak{A}_g(N)\rightarrow\cA_g(N)$. Furthermore, any point $a\in\cA_g(N)$ corresponds to the principally polarized abelian variety $(\mathfrak{A}_g(N)_a,\mathfrak{L}_g(N)_a)$ with some level-$N$-structure.
\begin{proof} See \cite[2.25, 3.21, 10.5, 10.10]{PinkThesis}.
\end{proof}
\end{thm}

\begin{notation} In the rest of the paper, we shall always take $N$ to be even and larger than 3. Furthurmoer we write $\cA_g$, $\mathfrak{A}_g$ and $\mathfrak{L}_g$ for $\cA_g(N)$, $\mathfrak{A}_g(N)$ and $\mathfrak{L}_g(N)$ for simplicity.
\end{notation}

\section{Weakly special subvarieties of $\mathfrak{A}_g$}\label{SectionWeaklySpecialSubvariety}
In this section, we discuss weakly special subvarieties of $\mathfrak{A}_g$ (or more generally, of mixed Shimura varieties of Kuga type).

\subsection{}
The following definition is not exactly the original one given by Pink \cite[Definition 4.1(b)]{PinkA-Combination-o}, but it is not hard to verify their equivalence (see \cite[Proposition~4.4(a)]{PinkA-Combination-o} and \cite[Proposition~5.7]{GaoTowards-the-And}):
\begin{defn}\label{WeaklySpecialDefinition}
A subvariety $Y$ of $\mathfrak{A}_g$ is called \textbf{weakly special} if there exist a connected mixed Shimura subdatum $(Q,\cY^+)$ of $(P_{2g,\mathrm{a}},\cX^+_{2g,\mathrm{a}})$, a connected normal subgroup $N$ of $Q$ possessing no non-trivial torus quotient and a point $\tilde{y}\in\cY^+$ such that $Y=\unif(N(\R)^+\tilde{y})$.
\end{defn}
\begin{rmk}
\begin{enumerate}
\item Weakly special subvarieties of $\mathfrak{A}_g$ defined as above are automatically irreducible (\cite[Remark 5.3]{GaoTowards-the-And}).
\item
For an arbitrary connected mixed Shimura variety $S$ of Kuga type, its weakly special subvarieties are defined in the same way with $(P_{2g,\mathrm{a}},\cX^+_{2g,\mathrm{a}})$ replaced by the connected mixed Shimura datum associated with $S$. For more general connected mixed Shimura varieties, the ``$N(\R)^+$" in the definition should be replaced by ``$N(\R)^+U_N(\C)$" where $U_N$ is the so-called weight $-2$ part of $N$. We shall not go into details on this.
\end{enumerate}
\end{rmk}

\subsection{}\label{GeometricInterpretationWeaklySpecial}
The goal of this subsection is to prove Proposition~\ref{DescriptionWeaklySpecial}. Recall that $P_{2g,\mathrm{a}}$ is defined to be $V_{2g}\rtimes\GSp_{2g}$ with the natural representation of $\GSp_{2g}$ on $V_{2g}$. Therefore this induces the zero-section $\epsilon\colon(\GSp_{2g},\H^+_g)\hookrightarrow(P_{2g,\mathrm{a}},\cX^+_{2g,\mathrm{a}})$ of $\pi$. Remark that $\epsilon$ corresponds to the zero-section of $[\pi]\colon\mathfrak{A}_g\rightarrow\cA_g$.

\begin{prop}\label{IntersectionWithWeaklySpecial}
Let $B$ be an irreducible subvariety of $\cA_g$ and $X:=[\pi]^{-1}(B)$. Define $\cC$ to be the isotrivial part of $X\rightarrow B$, i.e. the largest isotrivial abelian subscheme of $X$ over $B$. Then
\begin{align*}
\{\text{translates of an abelian subscheme of }X\rightarrow B\text{ by a torsion section and then }\\
\text{by a constant section of }\cC\rightarrow B\}
=\{X\cap E|~\text{E weakly special in }\mathfrak{A}_g\}.
\end{align*}
\end{prop}
The constant sections of $\cC\rightarrow B$ are defined as follows: By definition of isotriviality, there exists a finite cover $B^\prime\rightarrow B$ such that $\cC\times_{B}B^\prime\cong\cC_{b_0}\times B^\prime$ for any $b_0\in B$. A \textbf{constant section of $\cC\rightarrow B$} is then defined to be the image of the graph of a constant morphism $B^\prime\rightarrow\cC_{b_0}$ in $\cC\times_{B}B^\prime$ under the projection $\cC\times_{B}B^\prime\rightarrow\cC$.

It is clear that Proposition~\ref{DescriptionWeaklySpecial} follows immediately from Proposition~\ref{IntersectionWithWeaklySpecial} and \cite[4.3]{MoonenLinearity-prope}.

The following proposition is not hard to prove using Levi decomposition \cite[Theorem 2.3]{AlgebraicGroupBible}. Another (partial) proof can be found in \cite[Section 5.1]{LopuhaaPinks-conjectur}.

\begin{prop}\label{ShimuraSubDatum}
To give a Shimura subdatum $(Q,\cY^+)$ of $(P_{2g,\mathrm{a}},\cX^+_{2g,\mathrm{a}})$ is equivalent to giving:
\begin{itemize}
\item a pure Shimura subdatum $(G_Q,\cY^+_{G_Q})$ of $(\GSp_{2g},\H^+_g)$;
\item a $G_Q$-submodule $V_Q$ of $V_{2g}$ ($V_{2g}$ is a $\GSp_{2g}$-module, and therefore a $G_Q$-module);
\item an element $\bar{v}_0\in (V_{2g}/V_Q)(\Q)$.
\end{itemize}
\begin{proof} We only give the constructions here.
\begin{enumerate}
\item Given $(Q,\cY^+)\subset(P_{2g,\mathrm{a}},\cX^+_{2g,\mathrm{a}})$, we have $V_Q:=\cR_u(Q)<\cR_u(P_{2g,\mathrm{a}})=V_{2g}$. Therefore the inclusion $(Q,\cY^+)\subset(P_{2g,\mathrm{a}},\cX^+_{2g,\mathrm{a}})$ induces
\[
(G_Q,\cY^+_{G_Q}):=(Q,\cY^+)/V_Q\subset(\GSp_{2g},\H^+_g)=(P_{2g,\mathrm{a}},\cX^+_{2g,\mathrm{a}})/V_{2g}.
\]
The fact that $V_Q$ is a $G_Q$-submodule of $V_{2g}$ is clear. Now it suffices to find $\bar{v}_0\in(V_{2g}/V_Q)(\Q)$.

Consider the group $Q^\natural:=(V_{2g}/V_Q)\rtimes G_Q$, where the action is induced by the natural one of $G_Q$ on $V_{2g}$. By definition, $Q^\natural=\pi^{-1}(G_Q)/V_Q$. Now the inclusion $(Q,\cY^+)\subset(P_{2g,\mathrm{a}},\cX^+_{2g,\mathrm{a}})$ also induces an inclusion (which we call $i^\prime$)
\[
G_Q=Q/V_Q\subset \pi^{-1}(G_Q)/V_Q=Q^\natural.
\]
We have the following diagram, whose solide arrows commute:
\[
\begin{diagram}
1 &\rTo &1 &\rTo &G_Q &\rTo^{=} &G_Q &\rTo &1 \\
& &\dTo & &\dTo^{i^\prime} & & \dTo \\
1 &\rTo &V_{2g}/V_Q &\rTo &Q^\natural &\rTo\upperarrow{s_Q} &G_Q &\rTo &1
\end{diagram}
\]
where $s_Q$ is the homomorphism $G_Q=\{0\}\rtimes G_Q<(V_{2g}/V_Q)\rtimes G_Q=Q^\natural$. Now $i^\prime$ and $s_Q$ are two Levi-decompositions for $Q^\natural$. By \cite[Theorem~2.3]{AlgebraicGroupBible}, $s_Q$ equals the conjugation of $i^\prime$ by an element $\bar{v}_0\in(V_{2g}/V_Q)(\Q)$. Moreover, the choice of $\bar{v}_0$ is unique.

\item
Conversely, given the three data as in the Proposition, the underlying group $Q$ is the conjugate of $V_Q\rtimes G_Q<V_{2g}\rtimes\GSp_{2g}$ (compatible Levi-decompositions) by $(v_0,1)$ in $P_{2g,\mathrm{a}}$. The space
\[
\cY^+=\big(v_0+V_Q(\R)\big)\times\cY^+_{G_Q}\subset V_{2g}(\R)\times\H^+_g\cong\cX^+_{2g,\mathrm{a}}
\]
where $v_0$ is any lift of $\bar{v}_0$ to $V_{2g}(\Q)$.
\end{enumerate}
\end{proof}
\end{prop}

\begin{prop}\label{WeaklySpecialOnTheTopDescription}
A subvariety $Y$ of $\mathfrak{A}_g$ is weakly special iff there exist
\begin{itemize}
\item a pure Shimura subdatum $(G_Q,\cY^+_{G_Q})$ of $(\GSp_{2g},\H^+_g)$;
\item a point $v_0\in V_{2g}(\Q)$;
\item a normal semi-simple connected subgroup $G_N$ of $G_Q$ and a point $\tilde{y}_G\in\cY^+_{G_Q}$;
\item a $G_Q$-submodule $V_N$ of $V_{2g}$;
\item a $G_Q$-submodule $V_N^\perp$ of $V_{2g}$ on which $G_N$ acts trivially, and a point $v\in V_N^\perp(\R)$
\end{itemize}
such that
\[
Y=\unif\Big(\big(v_0+v+V_N(\R)\big)\times G_N(\R)^+\tilde{y}_G\Big).
\]
Here $\big(v_0+v+V_N(\R)\big)\times G_N(\R)^+\tilde{y}_G\subset V_{2g}(\R)\times\H^+_g\cong\cX^+_{2g,\mathrm{a}}$.
\begin{proof}
\begin{enumerate}
\item Given a weakly special subvariety $Y$ of $\mathfrak{A}_g$, let $(Q,\cY^+)$, $N$ and $\tilde{y}$ be as in Definition~\ref{WeaklySpecialDefinition}. By Proposition~\ref{ShimuraSubDatum}, $(Q,\cY^+)$ corresponds to a Shimura subdatum $(G_Q,\cY^+_{G_Q})$ of $(\GSp_{2g},\H^+_g)$, a $G_Q$-submodule $V_Q$ of $V_{2g}$ and a point $\bar{v}_0\in(V_{2g}/V_Q)(\Q)$. Let $v_0$ be any lift of $\bar{v}_0$ to $V_{2g}(\Q)$. Let $G_N:=N/(V_Q\cap N)$, then $G_N$ is a connected normal subgroup of $G_Q$, and hence is reductive. Since $N$ possesses no non-trivial torus quotient, $G_N$ is semi-simple. Let $\tilde{y}_G:=\pi(\tilde{y})$.

Let $V_N:=V_Q\cap N$, then $V_N$ is a $G_Q$-submodule of $V_Q$ since $N$ is normal in $Q$. By \cite[Corollary 2.14]{GaoTowards-the-And}, there exists a $G_Q$-submodule $V_N^\perp$ of $V_Q$ such that $V_Q=V_N\oplus V_N^\perp$ and $G_N$ acts trivially on $V_N^\perp$. Write $\tilde{y}=(\tilde{y}_V,\tilde{y}_G)\in(v_0+V_Q(\R))\times\cY^+_{G_Q}=\cY^+\subset\cX^+_{2g,\mathrm{a}}$ (here we use the second part of the proof of Proposition~\ref{ShimuraSubDatum}).

To simplify the computation below, we introduce a new Shimura subdatum $(Q^\prime,\cY^{\prime})$ of $(P_{2g,\mathrm{a}},\cX^+_{2g,\mathrm{a}})$: $(Q^\prime,\cY^{\prime})$ is defined to be the conjugate of $(Q,\cY^+)$ by $(-v_0,1)$. By the second part of the proof of Proposition~\ref{ShimuraSubDatum}, $(Q^\prime,\cY^{\prime})=(V_Q\rtimes G_Q,V_Q(\R)\times\cY^+_{G_Q})\subset(V_{2g}\rtimes\GSp_{2g},\cX^+_{2g,\mathrm{a}})$.
Let $N^\prime:=V_N\rtimes G_N<V_{2g}\rtimes\GSp_{2g}$, then $N^\prime$ is the conjugate of $N$ by $(-v_0,1)$. Let $\tilde{y}^\prime:=(\tilde{y}_V-v_0,\tilde{y}_G)\in\cY^{\prime+}$.

Let $v$ be the $V_N^\perp(\R)$-factor of $\tilde{y}_V$. Then since $G_N$ acts trivially on $V_N^\perp$, we have
\[
N^\prime(\R)^+\tilde{y}^\prime=\big(v+V_N(\R)\big)\times G_N(\R)^+\tilde{y}_G\subset\cY^{\prime+}.
\]
Hence $N(\R)^+\tilde{y}=\big(v_0+v+V_N(\R)\big)\times G_N(\R)^+\tilde{y}_G$.
Now the conclusion follows.

\item Conversely given all these data, let the Shimura subdatum $(Q,\cY^+)$ be the one obtained from $(G_Q,\cY^+_{G_Q})$, $V_N\oplus V_N^\perp$ and $v_0$ by Proposition~\ref{ShimuraSubDatum}. 
Let $N$ be the subgroup of $Q$ which is defined to be $V_N\rtimes G_N$ conjugated by $(v_0,1)$ in $P_{2g,\mathrm{a}}$. Then since $G_N$ acts trivially on $V_N^\perp$, $N\lhd Q$. Let $\tilde{y}:=(v_0+v,\tilde{y}_G)$. Now we have
\[
\big(v_0+v+V_N(\R)\big)\times G_N(\R)^+\tilde{y}_G=N(\R)^+\tilde{y}.
\]
The group $N$ is by definition connected and possessing no non-trivial torus quotient since $G_N$ is semi-simple. Hence $Y$ is weakly special by definition.
\end{enumerate}
\end{proof}
\end{prop}

Now we can prove Proposition~\ref{IntersectionWithWeaklySpecial}:
\begin{proof}[Proof of Proposition~\ref{IntersectionWithWeaklySpecial}]
\begin{enumerate}
\item Prove ``$\supset$". For this it suffices to prove:

\textit{~~For any weakly special subvariety $Y$ of $\mathfrak{A}_g$,
$Y$ is the translate of an abelian subscheme of $[\pi]^{-1}([\pi]Y)$ (over $[\pi]Y$) by a torsion section and then by a section of the isotrivial part of $[\pi]^{-1}[\pi]Y\rightarrow[\pi]Y$.}

Let $Y$ be a weakly special subvariety of $\mathfrak{A}_g$. Then associated to $Y$ there are data as in Proposition~\ref{WeaklySpecialOnTheTopDescription} and
\[
Y=\unif\Big(\big(v_0+v+V_N(\R)\big)\times G_N(\R)^+\tilde{y}_G\Big).
\]
Let $B^\prime:=[\pi]Y$ and $X^\prime:=[\pi]^{-1}(B^\prime)$.

Now $X^\prime\rightarrow B^\prime$ is an abelian scheme.
Since $V_N$ is a $G_Q$-submodule of $V_{2g}$, $\unif\big(V_N(\R)\times G_N(\R)^+\tilde{y}_G\big)$ is an abelian subscheme of $X^\prime$ over $B^\prime$. Therefore,
\[
\unif\Big(\big(v_0+V_N(\R)\big)\times G_N(\R)^+\tilde{y}_G\Big)
\]
is the translate of $B^\prime$ by a torsion section of $X^\prime\rightarrow B^\prime$. But $v\in V_N^\perp(\R)$ and $G_N$ acts trivially on $V_N^\perp$, so $\unif\big(V_N^\perp(\R)\times G_N(\R)^+\tilde{y}_G\big)$ is an isotrivial abelian scheme over $B^\prime$. Therefore $Y$ is the translate of an abelian subscheme of $X^\prime\rightarrow B^\prime$ by a torsion section and then by a section of the isotrivial part of $X^\prime\rightarrow B^\prime$.

\item Prove ``$\subset$". 
Let $Y$ be a subvariety of $X$ such that $Y$ is the translate of an abelian subscheme of $X\rightarrow B$ translated by a torsion section and then by a section of $\cC\rightarrow B$, where $\cC\rightarrow B$ is the isotrivial part of $X\rightarrow B$. Let us find a weakly special subvariety $E$ of $\mathfrak{A}_g$ associated with the data in Proposition~\ref{WeaklySpecialOnTheTopDescription} such that $Y=E\cap X$.

Let $B^\prime$ be the smallest weakly special subvariety of $\cA_g$ containing $B$. Then by definition there exist a Shimura subdatum $(G_Q,\cY^+_{G_Q})$, a connected semi-simple normal subgroup $G_N$ of $G_Q$ and a point $\tilde{y}_G\in\cY^+_{G_Q}$ such that $B^\prime=\unif_G\big(G_N(\R)^+\tilde{y}_G\big)$. Moreover by \cite[3.6, 3.7]{MoonenLinearity-prope}, $G_N$ is the connected algebraic monodromy group of $(B^{\prime})^{\sm}$, i.e. the neutral component of the Zariski closure of $\Gamma_{B^{\prime\sm}}:=$the image of $\pi_1((B^\prime)^{\sm})\rightarrow\pi_1(\cA_g)=\Gamma_G$.

Let $X^\prime:=[\pi]^{-1}(B^\prime)$. Then the isotrivial part $\cC^\prime$ of $X^\prime\rightarrow B^\prime$ is
\[
\unif\big(V^\prime(\R)\times G_N(\R)^+\tilde{y}_G\big),
\]
where $V^\prime$ is the largest $G_Q$-submodule of $V_{2g}$ on which $G_N$ acts trivially. This $V^\prime$ is the $V_N^\perp$ we want in Proposition~\ref{WeaklySpecialOnTheTopDescription}.

A key step is to prove that as subvarieties of $\mathfrak{A}_g$, we have
\begin{equation}\label{IsotrivialPartCompatibility}
\cC=\cC^\prime\cap X
\end{equation}
It is clear that $\cC^\prime\cap X\subset\cC$. For the other inclusion, suppose that $\cC$ is defined by the $G_Q$-submodule $V^{\prime\prime}$ of $V_{2g}$ (i.e. $\cC=\unif(V^{\prime\prime}(\R)\times\tilde{B})$ for $\tilde{B}:=\unif_G^{-1}(B)$), then $\Gamma_{B^{\prime\sm}}$ acts trivially on $V^{\prime\prime}$. However the action of $G$ on $V_{2g}$ is algebraic, therefore $\bar{\Gamma_{B^{\prime\sm}}}^{\Zar}$ acts trivially on $V^{\prime\prime}$. So $G_N$ acts trivially on $V^{\prime\prime}$. By the maximality of $V^\prime$, $V^{\prime\prime}\subset V^\prime$. So $\cC\subset\cC^\prime$. Now \eqref{IsotrivialPartCompatibility} follows.

Now since $Y$ is the translate of an abelian subscheme by a torsion section and then by a section of $\cC\rightarrow B$, there exists, by \eqref{IsotrivialPartCompatibility}, a $G_Q$-submodule $V_N$ of $V_{2g}$ such that
\[
Y=\unif\Big(\big(v_0+v+V_N(\R)\big)\times\tilde{B}\Big)
\]
where $v_0\in V_{2g}(\Q)$ corresponds to the torsion section and $v\in V^\prime(\R)$ corresponds to the section of $\cC\rightarrow B$.
In other words,
\[
Y=E\cap X\text{, where }E=\unif\Big(\big(v_0+v+V_N(\R)\big)\times G_N(\R)^+\tilde{y}_G\Big)
\]
and $E$ is the weakly special subvariety of $\mathfrak{A}_g$ we desire.
\end{enumerate}
\end{proof}

\subsection{Ax-Lindemann} In this subsection, we summarize some results regarding the mixed Ax-Lindemann theorem. All the results stated in this subsection hold for arbitrary connected mixed Shimura varieties, and in particular for $\mathfrak{A}_g$.

In this subsection, let $S$ be a connected mixed Shimura variety associated with $(P,\cX^+)$ and let $\unif\colon\cX^+\rightarrow S$ be the uniformization. An example for this is $\mathfrak{A}_g$ and $(P_{2g,\mathrm{a}},\cX^+_{2g,\mathrm{a}})$. As is explained in \cite[Proposition 4.1]{GaoTowards-the-And}, there exists a complex algebraic variety $\cX^\vee$, which is the total space of a holomorphic vector bundle (of rank $g$ in the case of $(P_{2g,\mathrm{a}},\cX^+_{2g,\mathrm{a}})$) over a complex projective variety, such that $\cX^+\hookrightarrow\cX^\vee$ makes $\cX^+$ a semi-algebraic\footnote{For any positive integer $N$, a semi-algebraic set of $\R^N$ is a subset defined by a finite sequence of $\R$-polynomial equations and inequalities, or any finite union of such sets.} and open (in the usual topology) subset of $\cX^\vee$.
\begin{defn}\label{AlgebraicityOnTheTop}
Let $\tilde{Y}$ be an analytic subvariety of $\cX^+$, then
\begin{enumerate}
\item $\tilde{Y}$ is called an irreducible algebraic subset of $\cX^+$ if it is an analytically irreducible component of the intersection of its Zariski closure in $\cX^\vee$ and $\cX^+$;
\item $\tilde{Y}$ is called algebraic if it is a finite union of irreducible algebraic subsets of $\cX^+$.
\end{enumerate}
\end{defn}

The following Ax-Lindemann theorem is due to Gao \cite{GaoTowards-the-And}:
\begin{thm}\label{Ax-Lindemann}
Let $\tilde{Z}$ be a semi-algebraic subset of $\cX^+$. Then any irreducible component of $\bar{\unif(\tilde{Z})}^{\Zar}$ is a weakly special subvariety of $S$.
\begin{proof} (see the forthcoming thesis \cite[Theorem~7.4]{GaoThesis})
Recall that a connected semi-algebraic subset of $\cX^+$ is called \textbf{irreducible} if its $\R$-Zariski closure in $\cX^\vee$ is an irreducible real algebraic variety. Note that any semi-algebraic subset of $\cX^+$ has only finitely many connected irreducible components. Let $\tilde{Z}^\prime$ be any connected irreducible component of $\tilde{Z}$. It suffices to prove that every irreducible component of $\bar{\unif(\tilde{Z})^\prime}^{\Zar}$ is weakly special.

Let $Y:=\bar{\unif(\tilde{Z}^\prime)}^{\Zar}$ and let $\tilde{W}$ be a connected irreducible semi-algebraic subset of $\cX^+$ which contains $\tilde{Z}^\prime$ and is contained in $\unif^{-1}(Y)$, maximal for these properties. Then
\[
Y=\bar{\unif(\tilde{W})}^{\Zar}.
\]
Now \cite[Lemma~4.1]{PilaAbelianSurfaces} claims that $\tilde{W}$ is algebraic in the sense of Definition~\ref{AlgebraicityOnTheTop}. Then any complex analytic irreducible component $\tilde{W}^\prime$ of $\tilde{W}$ is an irreducible algebraic subset of $\cX^+$ which is contained in $\unif^{-1}(Y)$, maximal for these properties. But then \cite[Theorem~1.2]{GaoTowards-the-And} tells us that $\unif(\tilde{W}^\prime)$ is a weakly special subvariety of $S$, and in particular a closed irreducible algebraic subvariety of $S$. Now $Y$ is the Zariski closure of $\unif(\tilde{W}^\prime)$ for $\tilde{W}^\prime$ running over the complex analytic irreducible components of $\tilde{W}$. Hence any irreducible component of $Y$ equals $\unif(\tilde{W}^\prime)$ for some $\tilde{W}^\prime$, and hence is a weakly special subvariety of $S$.
\end{proof}
\end{thm}

\section{Generalized Hecke orbit}\label{SectionGeneralizedHeckeOrbit}
In this section, we discuss the matrix expression of a polarized isogeny and then compute the generalized Hecke orbit of a point of $\mathfrak{A}_g$.

\subsection{Polarized isogenies and their matrix expressions}\label{SectionRationalRepresentationIsogeny}

Let $b\in\cA_g$. Denote by $A_b=\mathfrak{A}_{g,b}$ and denote by $\lambda_b\colon A_b\xrightarrow{\sim}A_b^\vee$ the principal polarization induced by $\mathfrak{L}_{g,b}$. Then the point $b$ corresponds to the polarized abelian variety $(A_b,\lambda_b)$. Let $\cB$ be a symplectic basis of $H_1(A_b,\Z)$ with respect to the polarization $\lambda_b$. Let $\tilde{b}\in\H_g^+$ be the period matrix of $A_b$ with respect to the basis $\cB$. In this subsection, we fix $\cB$ to be the $\Q$-basis of $V_{2g}$.

Consider all points $b^\prime\in\cA_g$ such that there exists a polarized isogeny
\[
f\colon (A_b,\lambda_b)\rightarrow(A_{b^\prime},\lambda_{b^\prime})
\]
where $(A_{b^\prime},\lambda_{b^\prime})=(\mathfrak{A}_{g,b^\prime},A_{b^\prime}\xrightarrow{\sim}A_{b^\prime}^\vee\text{ induced by }\mathfrak{L}_{g,b^\prime})$. Let $\cB^\prime$ be a symplectic basis of $H_1(A_{b^\prime},\Z)$ with respect to the polarization $\lambda_{b^\prime}$ and let $\tilde{b}^\prime\in\H_g^+$ be the period matrix of $A_{b^\prime}$ with respect to the basis $\cB^\prime$.

\begin{defn} The matrix $\alpha\in\GSp_{2g}(\Q)^+\cap\mathrm{M}_{2g\times 2g}(\Z)$ associated to
\[
f_*\colon H_1(A_b,\Z)\rightarrow H_1(A_{b^\prime},\Z)
\]
in terms of $\cB$ and $\cB^\prime$ is called the \textbf{rational representation of $f$} with respect to $\cB$ and $\cB^\prime$.
\end{defn}

The periods $\tilde{b}$ and $\tilde{b}^\prime$ are related by $\alpha$ in the following way:
\[
\tilde{b}=\alpha^t\cdot\tilde{b}^\prime=(A\tilde{b}^\prime+B)(C\tilde{b}^\prime+D)^{-1},\quad\text{ where }\alpha^t=\left(
\begin{array}{cc}
A &B \\
C &D \\
\end{array}\right)\text{ and }\tilde{b},\tilde{b}^\prime\in\H_g^+\subset M_{g\times g}(\C).
\]
Under the $\Q$-basis $\cB$ of $V_{2g}$, the matrix $\alpha^t$ corresponds to the dual isogeny of $f$, i.e. the following diagram commutes:
\begin{equation}\label{RationalRepresentationForfDual}
\begin{diagram}
(\cX^+_{2g,\mathrm{a}})_{\tilde{b}^\prime} &\rTo^{\alpha^t\cdot} &(\cX^+_{2g,\mathrm{a}})_{\tilde{b}} &,\quad &(v,\tilde{b}^\prime)\mapsto\left(\alpha^t v,\alpha^t\tilde{b}^\prime\right)=\left(\alpha^t v,\tilde{b}\right) \\
\dTo^{\unif} & &\dTo^{\unif} \\
A_{b^\prime} & &A_b\\
\dTo^{\lambda_b}_{\wr} & &\dTo^{\lambda_{b^\prime}}_{\wr} \\
A_{b^\prime}^\vee &\rTo^{f^\vee} &A_b^\vee
\end{diagram}.
\end{equation}

However, since $f$ is a polarized isogeny, $f^*\mathfrak{L}_{g,b^\prime}=\mathfrak{L}_{g,b}^{\otimes (\deg f)^{1/g}}$. So the following diagram commutes: 
\begin{equation}\label{CompositeOffAndfDual}
\begin{diagram}
A_b &\rTo^{f} &A_{b^\prime} \\
\dTo^{[(\deg f)^{1/g}]\circ\lambda_b} & &\dTo^{\lambda_{b^\prime}}_{\wr} \\
A_b^\vee &\lTo^{f^\vee} &A_{b^\prime}^\vee
\end{diagram}.
\end{equation}

Therefore by \eqref{RationalRepresentationForfDual} and \eqref{CompositeOffAndfDual}, we get the following commutative diagram:
\begin{equation}\label{RationalRepresentationPolarizedIsogeny}
\begin{diagram}
(\cX^+_{2g,\mathrm{a}})_{\tilde{b}} &\rTo^{(\deg f)^{1/g}(\alpha^t)^{-1}\cdot} &(\cX^+_{2g,\mathrm{a}})_{\tilde{b}^\prime} \\
\dTo^{\unif} & &\dTo^{\unif} \\
A_b &\rTo^f &A_{b^\prime}
\end{diagram}.
\end{equation}

\begin{defn} The matrix $(\deg f)^{1/g}(\alpha^t)^{-1}$ is called the \textbf{matrix expression of $f$} in coordinates $\cB$ with respect to $\cB^\prime$.
\end{defn}

\begin{rmk}
It is good to give the matrix $(\deg f)^{1/g}(\alpha^t)^{-1}$ a name because we will use it several times in the proof of Theorem~\ref{MainTheoremNonTorsion}. The name ``matrix expression" is given by the author. Remark that this definition only works for polarized isogenies because \eqref{CompositeOffAndfDual} fails for general non-polarized isogenies.
\end{rmk}

\subsection{Generalized Hecke orbit}
\begin{lemma}\label{AutomorphismGroupOfTheDatum}
Let $\varphi\in\aut\left((P_{2g,\mathrm{a}},\cX^+_{2g,\mathrm{a}})\right)$. Then there exist $g^\prime\in\GSp_{2g}(\Q)^+$ and $v_0\in V_{2g}(\Q)$ such that the action of $\varphi$ on $\cX^+_{2g,\mathrm{a}}$ is given by
\[
\varphi\left((v,x)\right)=(g^\prime v+v_0,g^\prime x).
\]
\begin{proof} We have $\varphi(V_{2g})=\varphi(\cR_u(P_{2g,\mathrm{a}}))\subset\cR_u(P_{2g,\mathrm{a}})=V_{2g}$. Since every two Levi decompositions of $P_{2g,\mathrm{a}}$ differs by the conjugation of an element $v_0\in V_{2g}(\Q)$, there exists a $v_0\in V_{2g}(\Q)$ such that $\psi:=\Int(v_0)^{-1}\circ\varphi$ maps $(\GSp_{2g},\H^+_g)$ to itself. Now $\psi$ maps $V_{2g}$ and $(\GSp_{2g},\H_g^+)$ to themselves. So $\psi$ can be written as $(A,B)$, where $A\in \GL_{2g}(\Q)$ and $B\in \aut\left((\GSp_{2g},\H^+_g)\right)=\GSp_{2g}(\Q)^+$. Remark that $\psi\in \aut(P_{2g,\mathrm{a}})$, so we can do the following computation: 

For any $v\in V_{2g}(\Q)$ and $h\in\GSp_{2g}(\Q)^+$,
\begin{align*}
(Ahv,BhB^{-1})=\psi((hv,h))=\psi((0,h)(v,1))=\psi(0,h)\psi(v,1) \\
=(0,BhB^{-1})(Av,1)=(BhB^{-1}Av,BhB^{-1}).
\end{align*}
Because $v$ is an arbitrary element of $V_{2g}(\Q)$, this implies that $Ah=BhB^{-1}A$ for any $h\in\GSp_{2g}(\Q)^+$. But this tells us that $A^{-1}B$ commutes with any element of $\GSp_{2g}(\Q)^+$, and hence $A^{-1}B\in\G_m(\Q)$. So $\psi$ acts on the group $P_{2g,\mathrm{a}}$ as $\psi((v,h))=(cBv,BhB^{-1})$ where $c\in\Q^*$ and $B\in\GSp_{2g}(\Q)^+$. Therefore $\psi$ acts on $\cX_{2g,a}^+$ as $\psi((v,x))=(cBv,Bx)=(cBv,cBx)$. Denote by $g^\prime:=cB\in\GSp_{2g}(\Q)^+$, then the action of $\varphi$ on $\cX^+_{2g,\mathrm{a}}$ is given by
\[
\varphi\left((v,x)\right)=(g^\prime v+v_0,g^\prime x).
\]
\end{proof}
\end{lemma}

Let $s\in\mathfrak{A}_g$, then $[\pi]s\in\cA_g$ corresponds to a polarized abelian variety $(\mathfrak{A}_{g,[\pi]s},\lambda_{[\pi]s})$.

\begin{cor}\label{GeneralizedHeckeIsogenyStep1}
Let $s\in\mathfrak{A}_g$. Then a point $t$ is in the generalized Hecke orbit of $s$ iff there exist a polarized isogeny $f\colon(\mathfrak{A}_{g,[\pi]s},\lambda_{[\pi]s})\rightarrow(\mathfrak{A}_{g,[\pi]t},\lambda_{[\pi]t})$ and $n^\prime\in\N$ such that $f(s)=n^\prime t$.
\begin{proof} Let $(v,x)\in\cX^+_{2g,\mathrm{a}}$ (resp. $(v_t,x_t)\in\cX^+_{2g,\mathrm{a}}$) be such that $s=\unif\left((v,x)\right)$ (resp. $t=\unif\left((v_t,x_t)\right)$). Then by Lemma~\ref{AutomorphismGroupOfTheDatum}, $t$ is in the generalized Hecke orbit of $s$ iff
\begin{equation}\label{GeneralizedHeckeCorrespondence}
(v_t,x_t)=(g^\prime v+v_0,g^\prime x)
\end{equation}
for some $g^\prime\in\GSp_{2g}(\Q)^+$ and $v_0\in V_{2g}(\Q)$.

If \eqref{GeneralizedHeckeCorrespondence} is satisfied, then there exists $c\in\G_m(\Q)=\Q^*$ s.t $h:=c^{-1}g^\prime\in\GSp_{2g}(\Q)^+$ is a $\Z$-coefficient matrix. Hence $h$ corresponds to a polarized isogeny $f\colon(\mathfrak{A}_{g,[\pi]s},\lambda_{[\pi]s})\rightarrow(\mathfrak{A}_{g,[\pi]t},\lambda_{[\pi]t})$. By \eqref{GeneralizedHeckeCorrespondence}, we have $t=\unif\left((chv+v_0,x_t)\right)$, and therefore
\[
n^\prime t=m^\prime f(s)+\unif\left((v_0,x_t)\right)
\]
where $c=m^\prime/n^\prime$. But $\unif\left((v_0,x_t)\right)$ is a torsion point of $\mathfrak{A}_{g,[\pi]t}$ since $v_0\in V_{2g}(\Q)$, and therefore can be removed by replacing $m^\prime$ and $n^\prime$ by sufficient large multiples. On the other hand $m^\prime f$ is still a polarized isogeny, and hnce replacing $f$ by $m^\prime f$, we may assume $m^\prime=1$. Finally we may assume $n^\prime\in\N$ by possibly replacing $f$ by $-f$.

On the other hand, suppose there exist a polarized isogeny $f\colon(\mathfrak{A}_{g,[\pi]s},\lambda_{[\pi]s})\rightarrow(\mathfrak{A}_{g,[\pi]t},\lambda_{[\pi]t})$ and $n^\prime\in\N$ such that $f(s)=n^\prime t$.
Let $\cB_s$ (resp. $\cB_t$) be a symplectic basis of $H_1(\mathfrak{A}_{g,[\pi]s},\Z)$ (resp. $H_1(\mathfrak{A}_{g,[\pi]t},\Z)$) and let $h$ be the matrix expression of $f$ in coordiante $\cB_s$ with respect to $\cB_t$. Then $h\in\GSp_{2g}(\Q)^+$ and there exists $(\gamma_V,\gamma_G)\in\Gamma$ such that
\[
(n^\prime v_t,x_t)=(\gamma_V,\gamma_G)(hv,hx)=(\gamma_V+\gamma_Ghv,\gamma_Ghx).
\]
Now $g^\prime:=\gamma_Gh/n^\prime \in\GSp_{2g}(\Q)^+$ and $v_0:=\gamma_V/n^\prime\in V_{2g}(\Q)$ satisfy \eqref{GeneralizedHeckeCorrespondence}.
\end{proof}
\end{cor}

\begin{cor}\label{GeneralizedHeckeIsogeny}
Let $s\in\mathfrak{A}_g$ and $t$ be a point in the generalized Hecke orbit of $s$. Let $f_t\colon(\mathfrak{A}_{g,[\pi]s},\lambda_{[\pi]s})\rightarrow(\mathfrak{A}_{g,[\pi]t},\lambda_{[\pi]t})$ be a polarized isogeny of minimal degree. Then there exist
\begin{itemize}
\item a point $s_0\in \mathfrak{A}_{g,[\pi]s}$;
\item $\varphi\in\End\left((\mathfrak{A}_{g,[\pi]s},\lambda_{[\pi]s})\right)$;
\item $n_0\in\N$
\end{itemize}
such that $s=n_0s_0$ and 
\[
f_t(\varphi(s_0)+p)=t
\]
for some torsion point $p\in\mathfrak{A}_{g,[\pi]s}$.
\begin{proof} By Corollary \ref{GeneralizedHeckeIsogenyStep1}, there exist a polarized isogeny $f\colon(\mathfrak{A}_{g,[\pi]s},\lambda_{[\pi]s})\rightarrow(\mathfrak{A}_{g,[\pi]t},\lambda_{[\pi]t})$ and $m^\prime,n^\prime\in\N$ such that $p_1:=m^\prime f(s)-n^\prime t$ is a torsion point of $\mathfrak{A}_{g,[\pi]t}$.
Now $f_t^{-1}\circ f\in\End\left((\mathfrak{A}_{g,[\pi]s},\lambda_{[\pi]s})\right)\otimes\Q$, i.e. there exist $\varphi^\prime\in\End\left((\mathfrak{A}_{g,[\pi]s},\lambda_{[\pi]s})\right)$ and $n_0^\prime\in\N$ such that $f_t^{-1}\circ f=\varphi^\prime\otimes(1/n_0)$. So $n_0^\prime\circ f=f_t\circ\varphi^\prime$ and hence
\[
m^\prime f_t(\varphi^\prime(s))=m^\prime n_0^\prime f(s)=n_0^\prime(n^\prime t+p_1)=n_0^\prime n^\prime t+n_0p_1.
\]
Let $\varphi:=m^\prime\circ\varphi^\prime\in\End\left((\mathfrak{A}_{g,[\pi]s},\lambda_{[\pi]s})\right)$ and $n_0:=n_0^\prime n^\prime\in\N$, then there exists a torsion point $p_2\in\mathfrak{A}_{g,[\pi]t}$ such that
\[
f_t(\varphi(s))=n_0t+p_2.
\]
Hence the conclusion follows.
\end{proof}
\end{cor}

\section{Diophantine estimate for the torsion case}\label{SectionDiophantineEstimateForTheTorsionCase}

\subsection{Preliminary}
In this subsection, we fix some definitions and notation used inminimal degree the proof of Theorem~\ref{MainTheoremZannier}.

Let $a\in\cA_g$. We use $\Sigma$ instead of $\Sigma_a$ to denote the set of all $a$-strongly special points of $\mathfrak{A}_g$. Let $\unif\colon \cX^+_{2g,\mathrm{a}}\rightarrow\mathfrak{A}_g$ be the uniformization map and let $\cF$ be the fundamental set in $\cX^+_{2g,\mathrm{a}}$ defined as in Theorem~\ref{UniversalFamilyTheorem}.(3). Let
\[
\tilde{Y}:=\unif^{-1}(Y)\cap\cF \text{ and } \tilde{\Sigma}:=\unif^{-1}(\Sigma)\cap\cF.
\]
The point $a\in\cA_g$ corresponds to the polarized abelian variety $(A_a,\lambda_a):=(\mathfrak{A}_{g,a},\lambda_a)$. Let $\cB$ be a symplectic basis for $H_1(A_a,\Z)$ with respect to the polarization $\lambda_a$. Let $\tilde{a}$ be the period matrix of $A_a$ with respect to the chosen basis $\cB$. In the rest of the paper, we shall sometimes identify $\tilde{a}\in\H_g^+$ and $(0,\tilde{a})\in\{0\}\times\H_g^+\subset V_{2g}(\R)\times\H_g^+\cong\cX^+_{2g,\mathrm{a}}$.

For any $t\in\Sigma$, there exists by definition of $\Sigma_a$ a polarized isogeny $(A_a,\lambda_a)\rightarrow(\mathfrak{A}_{g,[\pi]t},\lambda_{[\pi]t})$. Besides, $t$ is a torsion point of $A_{[\pi]t}:=\mathfrak{A}_{g,[\pi]t}$, whose order we denote by $N(t)$.

\begin{defn} For any $t\in\Sigma$, define its \textbf{complexity} to be
\[
\max\left(\text{minimal degree of polarized isogenies }(A_a,\lambda_a)\rightarrow(A_{[\pi]t},\lambda_{[\pi]t}),N(t)\right).
\]
In addition, define the \textbf{complexity} of any point of $\tilde{\Sigma}$ to be the complexity of its image in $\Sigma$.
\end{defn}

\subsection{Application of Pila-Wilkie}
The goal of this subsection is to prove the following proposition:
\begin{prop}\label{PilaWilkieForComplexity}
Let $Y$, $\tilde{a}$ be as in the last subsection. Let $\epsilon>0$. There exists a constant $c=c(Y,\tilde{a},\epsilon)>0$ with the following property:

For every $n\geqslant 1$, there exist at most $cn^{\epsilon}$ definable blocks $B_i\subset\tilde{Y}$ such that $\cup B_i$ contains all points of complexity at most $n$ in $\tilde{Y}\cap\tilde{\Sigma}$. 
\end{prop}

\begin{lemma}\label{TranslateWRTComplexity}
There exist constants $c^\prime$, $\kappa$ depending only on $g$ and $\tilde{a}$ such that

For any $\tilde{t}\in\tilde{Y}\cap\tilde{\Sigma}$ of complexity $n$, there exists a $(v,h)\in P_{2g}(\Q)^+$ such that $(v,h)\tilde{a}=\tilde{t}$ and $H((v,h))\leqslant c^\prime n^{\kappa}$.
\begin{proof} Let $t=\unif(\tilde{t})$.  By \cite[Proposition 4.1]{OrrFamilies-of-abe}, there exist
\begin{itemize}
\item a polarized isogeny $f\colon\mathfrak{A}_{g,[\pi]t}\rightarrow A_a$;
\item a symplectic basis $\cB^\prime$ for $H_1(\mathfrak{A}_{g,[\pi]t},\Z)$ with respect to the polarization $\lambda_{[\pi]t}$
\end{itemize}
such that the rational representation $h_1$ of $f$ with respect to the chosen bases satisfies that $H(h_1)$ is polynomially bounded in $\deg(f)$.

But $\unif_G(h_1^t\tilde{a})=[\pi]t$ by \eqref{RationalRepresentationPolarizedIsogeny}.
Hence there exists a $h_2\in\Gamma_G$ such that $h_2h_1^t\tilde{a}=\pi(\tilde{t})\in\cF_G$. By \cite[Lemma 3.2]{PilaAbelianSurfaces}, $H(h_2)$ is polynomially bounded in the norm of $h_1^t\cdot\tilde{a}$. 

Now define $h:=h_2h_1^t$. We have then $h\tilde{a}=\pi(\tilde{t})$ and
\[
H(h)\leqslant c_0\deg(f)^{\kappa_0}
\]
where $c_0>0$ and $\kappa_0>0$ depend only on $g$ and $\tilde{a}$.

Next write $\tilde{t}=(\tilde{t}_V,\pi(\tilde{t}))\in\cF$. Let $v:=\tilde{t}_V$, then $v\in V_{2g}(\Q)$ since $t$ is a torsion point of $\mathfrak{A}_{g,[\pi]t}$. Besides, the denominator of $v$ is precisely the order of the torsion point $t$. But by choice, $\cF\cong [0,N)^{2g}\times\cF_G\subset V_{2g}(\R)\times\H_g^+\cong\cX^+_{2g,\mathrm{a}}$ (see Theorem~\ref{UniversalFamilyTheorem}.(3)). Therefore up to a constant depending on nothing, $H(v)$ is bounded by its denominator, i.e. the order of the torsion point $t$ of $\mathfrak{A}_{g,[\pi]t}$.

To sum it up, $(v,h)$ is the element of $P_{2g}(\Q)^+$ which we dezire.
\end{proof} 
\end{lemma}

Now we can prove Proposition~\ref{PilaWilkieForComplexity} with the help of Lemma~\ref{TranslateWRTComplexity}.
\begin{proof}[Proof of Proposition~\ref{PilaWilkieForComplexity}]
Let
\begin{align*}
\sigma\colon P_{2g}(\R)^+ &\rightarrow\cX^+_{2g,\mathrm{a}} \\
(v,h) &\mapsto (v,h)\tilde{a}
\end{align*}

The set $R:=\sigma^{-1}(\tilde{Y})=\sigma^{-1}(\unif^{-1}(Y)\cap\cF)$ is definable because $\sigma$ is semi-algebraic and $\unif|_{\cF}$ is definable. Hence we can apply the family version of the Pila-Wilkie theorem (\cite[3.6]{PilaO-minimality-an}) to the definable set $R$: for every $\epsilon>0$, there are only finitely many definable block families $B^{(j)}(\epsilon)\subset R\times\R^m$ and a constant $C_1(R,\epsilon)$ such that for every $T\geqslant 1$, the rational points of $R$ of height at most $T$ are contained in the union of at most $C_1T^\epsilon$ definable blocks $B_i(T,\epsilon)$, taken (as fibers) from the families $B^{(j)}(\epsilon)$. Since $\sigma$ is semi-algebraic, the image under $\sigma$ of a definable block in $R$ is a finite union of definable blocks in $\tilde{Y}$. Furthermore the number of blocks in the image is uniformly bounded in each definable block family $B^{(j)}(\epsilon)$. Hence $\sigma(B_i(T,\epsilon))$ is the union of at most $C_2T^{\epsilon}$ blocks in $\tilde{Y}$, for some new constant $C_2(Y,\tilde{a},\epsilon)>0$.

By Lemma~\ref{TranslateWRTComplexity}, for any point $\tilde{t}\in\tilde{Y}\cap\tilde{\Sigma}$ of complexity $n$, there exists a rational element $\gamma\in R$ such that $\sigma(\gamma)=\tilde{t}$ and $H(\gamma)\leqslant c^\prime n^\kappa$. By the discussion in the last paragraph, all such $\gamma$'s are contained in the union of at most $C_1(c^\prime n^\kappa)^{\epsilon}$ definable blocks. Therefore all points of $\tilde{Y}\cap\tilde{\Sigma}$ of complexity $n$ are contained in the union of at most $C_1C_2c^{\prime\epsilon}n^{\kappa\epsilon}$ blocks in $\tilde{Y}$.
\end{proof}

\subsection{Galois orbit}
In this section we shall deal with the Galois orbit. We handle the case of $\bar{\Q}$-points at first and then use the standard specialization argument to prove the result for general points of $\Sigma\cap Y$.

\begin{prop}\label{GaloisOrbitNumberField}
Suppose $a\in\cA_g(\bar{\Q})$. There exist positive constants $c_1^\prime=c_1^\prime(g)$, $c_2^\prime=c_2^\prime(g,k(a))$ and $c_3^\prime=c_3^\prime(g)$ satisfying the following property:

For any point $t\in\Sigma\cap Y\cap\mathfrak{A}_g(\bar{\Q})$ of complexity $n$,
\[
[k(t):\Q]\geqslant c_1^\prime\frac{n^{c_2^\prime}}{h_F(A_a)^{c_3^\prime}}
\]
where $k(t)$ is the definition field of $t$.
\begin{proof} Define (as Gaudron-R\'{e}mond \cite{GaudronPolarisations-e})
\[
\kappa(\mathfrak{A}_{g,[\pi]t}):=((14g)^{64g^2}[k([\pi]t):\Q]\max(h_F(\mathfrak{A}_{g,[\pi]t}),\log[k([\pi]t):\Q],1)^2)^{1024g^3}.
\]
Take a point $t\in\Sigma\cap Y\cap\mathfrak{A}_g(\bar{\Q})$ of complexity $n$. Denote by $k([\pi]t)$) the definition field of $[\pi]t$. Denote by $N(t)$ the order of $t$ as a torsion point of $A_{[\pi]t}:=\mathfrak{A}_{g,[\pi]t}$. There are two cases.

$\framebox{\textit{Case i}}$ $n=\text{minimal degree of polarized isogenies }(A_a,\lambda_a)\rightarrow(A_{[\pi]t},\lambda_{[\pi]t})$. Then by \cite[Th\'eor\`eme 1.4]{GaudronPolarisations-e} and \cite[Theorem 5.6]{MartinThesis},
\[
n\leqslant\kappa(\mathfrak{A}_{g,[\pi]t}).
\]
On the other hand, by a result of Faltings \cite[Chapter II, $\mathsection$4, Lemma 5]{CornellArithmetic-Geom},
\[
h_F(\mathfrak{A}_{g,[\pi]t})\leqslant h_F(A_a)+(1/2)\log n.
\]
Now the conclusion for this case follows from the two inequalities above and the easy fact $[k(t):\Q]\geqslant[k([\pi]t):\Q]$.

$\framebox{\textit{Case ii}}$ $n=N(t)$. By \cite[Th\'{e}or\`{e}me 1.2]{GaudronPolarisations-e}, there exist positive natural numbers $l$, simple abelian varieties $A_1$,...,$A_l$ over a finite extension $k^\prime$ of $k([\pi]t)$ ($A_i$ and $A_j$ can be isogenous to each other over $\bar{\Q}$ for $i\neq j$) and an isogeny
\begin{equation}\label{IsogenyOfGaudronRemond}
\varphi\colon \mathfrak{A}_{g,[\pi]t}\rightarrow\prod_{i=1}^lA_i
\end{equation}
such that $\varphi$ is defined over $k^\prime$, $\deg\varphi\leqslant\kappa(\mathfrak{A}_{g,[\pi]t})$ and $[k^\prime:k([\pi]t)]\leqslant\kappa(\mathfrak{A}_{g,[\pi]t})^g$. Call $p_i\colon A\rightarrow A_i$ the composite of $\varphi$ and the $i$-th projection $\prod_{i=1}^lA_i\rightarrow A_i$ ($\forall i=1,...,l$).

Now $t\in A$ is a torsion point of order $\mathfrak{A}_{g,[\pi]t}$. Without any loss of generality we have
\[
N(p_1(t))\geqslant N(p_i(t))
\]
where $N(p_i(t))$ is the order of $p_i(t)$ as a torsion point of $A_i$.

\begin{lemma}\label{CompareTorsionOrderWithSimpleQuotients}
\[
N(t)\leqslant \kappa(\mathfrak{A}_{g,[\pi]t})N(p_1(t))^g\text{ and }[k(t):\Q]\geqslant[k(p_1(t)):\Q]/\kappa(\mathfrak{A}_{g,[\pi]t})^{2g}.
\]
where $k(p_1(t))$ is the definition field of $p_1(t)$.
\begin{proof} Denote by $N(\varphi(t))$ the order of $\varphi(t)$ as a torsion point of $\prod_{i=1}^lA_i$. It is clear that
\[
N(\varphi(t))\geqslant N(t)/\deg\varphi\geqslant N(t)/\kappa(\mathfrak{A}_{g,[\pi]t}).
\]
On the other hand, $N(\varphi(t))=\text{lcd}(N(p_1(t)),...,N(p_l(t)))\leqslant N(p_1(t))^g$. Now the first inequality follows.

For the second inequality, first of all since $\varphi$ and $\prod_{i=1}^lA_i$ are both defined over $k^\prime$, we have
\[
[k(\varphi(t)):\Q]\leqslant[k(t)k^\prime:\Q]=[k(t):\Q][k(t)k^\prime:k(t)]\leqslant[k(t):\Q][k^\prime:k]\leqslant[k(t):\Q]\kappa(\mathfrak{A}_{g,[\pi]t})^g.
\]
Next since all abelian varieties $A_1$,...,$A_l$ are defined over $k^\prime$, we have then
\[
[k(\varphi(t))k^\prime:\Q]\geqslant[k(p_1(t)):\Q].
\]
But
\begin{align*}
[k(\varphi(t))k^\prime:\Q] &=[k(\varphi(t))k^\prime:k^\prime][k^\prime:k][k:\Q] \\
& \leqslant[k(\varphi(t)):k][k^\prime:k][k:\Q] \\
& =[k(\varphi(t)):\Q][k^\prime:k] \\
& \leqslant[k(\varphi(t)):\Q]\kappa(\mathfrak{A}_{g,[\pi]t})^g.
\end{align*}
Now the second inequality follows from the three inequalities above.
\end{proof}
\end{lemma}

By \cite[Corollaire 1.5]{DavidMinorations-de-},
\begin{equation}\label{SinnouDavidOriginalForTorsionPoint}
[k(p_1(t)):\Q]\geqslant c_0^\prime(g)\frac{N(p_1(t))^{1/(2g)}}{\log N(p_1(t))(h_F(A_1)+\log N(p_1(t)))}.
\end{equation}
By the comment below \cite[Corollaire 1.5]{GaudronPolarisations-e}, we may assume
\begin{equation}\label{GoodChoiceOfSimpleAbelianVarieties}
h_F(A_1)\leqslant h_F(\mathfrak{A}_{g,[\pi]t})+\frac{1}{2}\log\kappa(\mathfrak{A}_{g,[\pi]t}).
\end{equation}

By assumption of this case, there exists an isogeny $A_a\rightarrow\mathfrak{A}_{g,[\pi]t}$ of degree $\leqslant n$. So by Faltings \cite[Chapter II, $\mathsection$4, Lemma 5]{CornellArithmetic-Geom},
\begin{equation}\label{ComparisonOfFaltingsHeightUnderIsogeny}
h_F(\mathfrak{A}_{g,[\pi]t})\leqslant h_F(A_a)+(1/2)\log n.
\end{equation}

Now because $[k(t):\Q]\geqslant[k([\pi]t):\Q]$, the conclusion of \textit{Case ii} now follows from Lemma~\ref{CompareTorsionOrderWithSimpleQuotients}, \eqref{SinnouDavidOriginalForTorsionPoint}, \eqref{GoodChoiceOfSimpleAbelianVarieties} and \eqref{ComparisonOfFaltingsHeightUnderIsogeny}.
\end{proof}
\end{prop}

\begin{cor}\label{GaloisOrbitGeneralCase}
Suppose $a$ is defined over a finitely generated field $k$.
There exist positive constants $c_1=c_1(A_a,k)$ and $c_2=c_2(A_a,k)$ satisfying the following property:

For any point $t\in\Sigma\cap Y$ of complexity $n$ defined over a finitely extension $k(t)$ of $k$,
\[
[k(t):k]\geqslant c_1n^{c_2}.
\]
\begin{proof} This follows from Proposition~\ref{GaloisOrbitNumberField} and a specialization argument. The case where $n=\text{minimal degree of polarized isogenies }(A_a,\lambda_a)\rightarrow(A_{[\pi]t},\lambda_{[\pi]t})$ is proved by Orr \cite[Theorem 5.1]{OrrFamilies-of-abe} (possibly combined with \cite[Theorem 5.6]{MartinThesis}). The case where $n=N(t)$, the order of $t$ as a torsion point of $\mathfrak{A}_{g,[\pi]t}$, follows from the standard specialization argument introduced by Raynaud (see \cite[Section 5]{OrrFamilies-of-abe} or \cite[Section 7]{RaynaudCourbes-sur-une}).
\end{proof}
\end{cor}

\section{End of the proof in the torsion case}\label{SectionEndOfTheProofForTheTorsionCase}
In this section, $Y$ is always an irreducible subvariety of $\mathfrak{A}_g$, $a\in\cA_g$ and $\Sigma$ is the set of all $a$-strongly special points of $\mathfrak{A}_g$.
\begin{thm}\label{WeaklySpecialZariskiDense}
If $\bar{Y\cap\Sigma}^{\Zar}=Y$, then the union of all positive-dimensional weakly special subvarieties contained in $Y$ is Zariski dense in $Y$.
\begin{proof} Let $\Sigma_1$ be the set of points $t\in Y\cap\Sigma$ such that there is a positive-dimensional block $B\subset\tilde{Y}$ with $t\in \unif(B)$. Let $Y_1$ be the Zariski closure of $\Sigma_1$. Let $k$ be the finitely generated field $k(a)$. Enlarge $k$ if necessary such that both $Y$ and $Y_1$ are defined over $k$.

Let $t$ be a point in $Y\cap\Sigma$ of complexity $n$. By Corollary \ref{GaloisOrbitGeneralCase}, there exist positive constants $c_1$ and $c_2$ depending only on $g$, $A_a$ and $k$ such that
\[
[k(t):k]\geqslant c_1n^{c_2/2}.
\]

But all $\gal(\bar{k}/k)$-conjugates of $t$ are contained in $Y\cap \Sigma$ and have complexity $n$. By Proposition~\ref{PilaWilkieForComplexity}, the preimages in $\cF$ of these points are contained in the union of $c(Y,\tilde{a},c_2/4)n^{c_2/4}$ definable blocks, each of these blocks being contained in $\tilde{Y}$.

For $n$ large enough, $c_1n^{c_2/2}>cn^{c_2/4}$. Hence for $n\gg0$, there exists a definable block $B\subset\tilde{Y}$ such that $\unif(B)$ contains at least two Galois conjugates of $t$, and therefore $\dim B>0$ since blocks are connected. So being in $\unif(B)$, those conjugates of $t$ are in $\Sigma_1$. But $Y_1$ is defined over $k$, so $t\in Y_1$.

In summary, all points of $Y\cap\Sigma$ of large enough complexity are in $\Sigma_1$. This excludes only finitely many points of $Y\cap\Sigma$. So $Y_1=Y$.

Let $\Sigma_2$ be the set of points $t\in Y\cap\Sigma$ such that there is a positive-dimensional connected semi-algebraic set $B^\prime\subset\tilde{Y}$ with $t\in \unif(B^\prime)$. Let $Y_2$ be the Zariski closure of $\Sigma_2$. By definition of blocks, $\Sigma_2=\Sigma_1$, and hence $Y_2=Y_1=Y$. But the Ax-Lindemann theorem (in the form of Theorem~\ref{Ax-Lindemann}) implies that the irreducible component $Z$ of $\bar{\unif(B^\prime)}^{\Zar}$ containing $t$ is weakly special. Moreover $\dim(Z)>0$ since $\dim(B^\prime)>0$. Therefore every point $t\in\Sigma_2$ is contained in some positive-dimensional weakly special subvariety of $\mathfrak{A}_g$. Now the conclusion follows. 
\end{proof}
\end{thm}

\begin{proof}[Proof of Theorem~\ref{MainTheoremZannier}]
Let $S$ be the smallest connected mixed Shimura subvariety containing $Y$. Assume $S$ is associated with the connected mixed Shimura datum $(P,\cX^+)$. Let $(G,\cX^+_G):=(P,\cX^+)/\cR_u(P)$. By Theorem~\ref{WeaklySpecialZariskiDense} and \cite[Theorem 12.2]{GaoTowards-the-And}, such a non-trivial group $N$ exists: $N$ is the maximal normal subgroup of $P$ such that the followings hold:
\begin{itemize}
\item there exists a diagram of Shimura morphisms
\[
\begin{diagram}
(P,\cX^+) &\rTo^{\rho} &(P^\prime,\cX^{\prime+}):=(P,\cX^+)/N &\rTo^{\pi^\prime} &(G^\prime,\cX^{\prime+}_G):=(P^\prime,\cX^{\prime+})/\cR_u(P^\prime) \\
\dTo^{\unif} & &\dTo^{\unif^\prime}  & &\dTo^{\unif^\prime_G}\\
S &\rTo^{[\rho]} &S^\prime &\rTo^{[\pi^\prime]} &S^\prime_G
\end{diagram}
\]
(then $S^\prime$ is by definition a connected Shimura variety of Kuga type)\\
\item the union of positive-dimensional weakly special subvarieties which are contained in $Y^\prime:=\bar{[\rho](Y)}$ is not Zariski dense in $Y^\prime$; \\
\item $Y=[\rho]^{-1}(Y^\prime)$.
\end{itemize}

\begin{enumerate}
\item
We prove the theorem by induction on $g$. When $g=1$, the only non-trivial case is when $Y$ is a curve. But then $Y$ must be weakly special by Theorem~\ref{WeaklySpecialZariskiDense}. Remark that this case has also been proved by Andr\'{e} \cite[Lecture 4]{AndreShimura-varieti} when he proposed the mixed Andr\'{e}-Oort conjecture.

When $\dim([\pi](Y))=0$, this is the Manin-Mumford conjecture by Corollary \ref{GeneralizedHeckeIsogeny}. Hence we only have to deal with the case $\dim([\pi](Y))=1$. Remark that in this case $[\pi](Y)$ is weakly special by the main result of \cite{OrrFamilies-of-abe}, and hence equals $\unif_G\left(G^{\prime\prime}(\R)^+\tilde{y}\right)$ for some $G^{\prime\prime}<\GSp_{2g}$ of positive dimension and $\tilde{y}\in\H_g^+$. Now there are two cases:

If $\dim([\pi^\prime](Y^\prime))=0$, then $[\pi^\prime](Y^\prime)$ is a point. In this case $Y^\prime$ is a subvariety of an abelian variety. The hypothesis $\bar{Y\cap\Sigma}=Y$ implies that $Y^\prime$ contains a Zariski dense subset of torsion points. Therefore by the result of the Manin-Mumford conjecture, $Y^\prime$ is a special subvariety, i.e. the translate of an abelian subvariety by a torsion point. But the union of positive-dimensional weakly special subvarieties which are contained in $Y^\prime:=\bar{[\rho](Y)}$ is not Zariski dense, so $Y^\prime$ is a point. Therefore $Y$ is weakly special by definition.

If $\dim([\pi^\prime](Y^\prime))=1$, then $N/\cR_u(N)$ is trivial because the dimension of $[\pi](Y)=\unif_G\left(G^{\prime\prime}(\R)^+\tilde{y})\right)$ is $1$. Therefore $V_N:=\cR_u(N)<V_{2g}$ is non-trivial since $N$ is non-trivial.

Denote for simplicity by $B:=[\pi^\prime](Y^\prime)=\unif^\prime_G(G^{\prime\prime}(\R)^+\rho(\tilde{y}))$ and $X:=[\pi^\prime]^{-1}(B)$. Then $X\rightarrow B$ is a family of abelian varieties of dimension $g^\prime$. We have $g^\prime<g$ since $V_N$ is non-trivial. Besides, $X\rightarrow B$ is non-isotrivial because otherwise $G^{\prime\prime}$ acts trivially on $V_{2g}/V_N$, and therefore $G^{\prime\prime}\lhd P^\prime$. This contradicts the maximality of $N$. Hence there exists, up to taking finite covers of $X\rightarrow B$, a cartesian diagram
\[
\begin{diagram}
X &\rTo^i &\mathfrak{A}_{g^\prime} \\
\dTo & &\dTo \\
B &\rTo^{i_B} &\cA_{g^\prime}
\end{diagram}
\]
such that both $i$ and $i_B$ are finite. Apply induction hypothesis to $i(Y^\prime)\subset\mathfrak{A}_{g^\prime}$, we get that $i(Y^\prime)$ is weakly special. By the geometric interpretation of weakly special subvarieties (Proposition~\ref{DescriptionWeaklySpecial}), $i^{-1}(i(Y^\prime))$ is irreducible. Therefore $Y^\prime=i^{-1}(i(Y^\prime))$ since they are of the same dimension. So $Y^\prime$ is a weakly special subvariety of $S^\prime$ (again by Proposition~\ref{DescriptionWeaklySpecial}). But then $Y^\prime$ must be a point by definition of $Y^\prime$. Hence $Y$ is weakly special by definition.

\item
This part of the theorem is the intersection of the Andr\'{e}-Oort conjecture and Conjecture~\ref{AndrePinkConjecture}. It holds in a more general situation (see the forthcoming thesis \cite[Theorem~4.3.2]{GaoThesis}). The proof, which requires more background knowledge about mixed Shimura varieties, is similar to \cite[Theorem 13.6]{GaoTowards-the-And}, except that the lower bound used in that article is replaced by a result similar to (but weaker than) Corollary \ref{GaloisOrbitGeneralCase}. More explicitly:

Since $a\in\cA_g$ is a special point, every $a$-strongly special point is a special point of $\mathfrak{A}_g$. Therefore $Y^\prime$ contains a Zariski dense subset of special points. Besides, $Y$ is $a$-special iff $Y$ is a special subvariety of $\mathfrak{A}_g$ by Proposition~\ref{DescriptionWeaklySpecial}.

Suppose that $Y$ is not $a$-special. Then $Y^\prime$ is not a special subvariety of $S^\prime$. On the other hand, $Y^\prime$ is defined over a number field since every point in $\Sigma_a^\prime$ is.

Define $V_N:=\cR_u(N)<V_{2g}$ and $G_N:=N/V_N\lhd G<\GSp_{2g}$. The reductive group $G$ decomposes as an almost direct product $Z(G)H_1...H_r$ with all $H_i$'s simple. Without any loss of generality, we may assume that $H_1$,...,$H_l$ are the simple factors of $G$ which appear in the decomposition of $G_N$. Define $G_N^\perp:=H_{l+1}...H_r$. Define $T:=\MT(a)$, then $T$ is a torus since $a$ is a special point of $\cA_g$.

Let $G_1:=G_N^\perp T$. This is a subgroup of $G$ (and therefore a subgroup of $\GSp_{2g}$). Moreover, it defines a connected Shimura subdatum $(G_1,\cX^+_{G_1})$ of $(\GSp_{2g},\H_g^+)$ and hence its associated connected Shimura subvariety $S_{G_1}$ of $\cA_g$ such that $a\in S_{G_1}$. Recall that $(P^\prime,\cX^{\prime+})=(P,\cX^+)/N$ and $(G^\prime,\cX_G^{\prime+})=(G,\cX^+_G)/G_N$. Therefore the natural Shimura morphisms
\[
(G_1,\cX^+_{G_1})\hookrightarrow(G,\cX^+_G)\twoheadrightarrow(G^\prime,\cX_G^{\prime+})
\]
identify $\cX^+_{G_1}$ and $\cX_G^{\prime+}$.

Recall that $P=V\rtimes G$ gives rises to a connected mixed Shimura datum $(P,\cX^+)$. So $V:=\cR_u(P)$ is a $G_1$-module such that the action of $G_1$ on $V$ induces a Hodge-structure of type $\{(-1,0),(0,-1)\}$ on $V$. Therefore by \cite[2.17]{PinkThesis}, there exists a connected mixed Shimura datum $(P_1,\cX^+_1)$ such that $P_1=V\rtimes G_1$ and $(G_1,\cX_{G_1})=(P_1,\cX^+_1)/V$. $(P_1,\cX^+_1)$ is a connected mixed Shimura subdatum of $(P,\cX^+)$. Since $N\lhd P$, we have $V_N\lhd P_1$. Now we have the following diagram of Shimura morphisms:
\[
\begin{diagram}
(P_2,\cX^+_2):=(P_1,\cX^+_1)/V_N &\lOnto^{\rho^\prime} &(P_1,\cX^+_1) &\rInto^j &(P,\cX^+) &\rTo^{\rho} &(P^\prime,\cX^{\prime+}) \\
\dTo_{\unif_2} & &\dTo & &\dTo & &\dTo_{\unif^\prime} \\
S_2 &\lOnto^{[\rho^\prime]} &S_1 &\rTo^{[j]} &S &\rTo^{[\rho]} &S^\prime
\end{diagram}.
\]
Then the map $\rho\circ j\circ\rho^{\prime-1}\colon(P_2,\cX^+_2)\rightarrow(P^\prime,\cX^{\prime+})$ is well-defined and is a Shimura morphism. Hence $Y^\prime$ is a special subvariety of $S^\prime$ iff $Y_2:=([\rho]\circ[j]\circ[\rho^\prime]^{-1})^{-1}(Y^\prime)$ is a special subvariety of $S_2$. Hence it suffices to prove that $Y_2$ is special. But $\cX^+_2$ and $\cX^{\prime+}$ are identified under $\rho\circ j\circ\rho^{\prime-1}$ by the discussion in the last paragraph, so the union of positive-dimensional weakly special subvarieties of $Y_2$ is not Zariski dense in $Y_2$ by choice of $Y^\prime$. Therefore we are left to prove that the set of special points of $Y_2$ which do not lie in any positive-dimensional special subvariety is finite. Remark that $Y_2$ is defined over a number field (which we call $k$) since $Y^\prime$ is.

Take the pure part of the diagram above, we get the following diagram of Shimura morphisms between pure Shimura data and pure Shimura varieties:
\[
\begin{diagram}
(G_2,\cX^+_{G_2}) &\lTo^{\rho^\prime_G}_{\sim} &(G_1,\cX^+_{G_1}) &\rInto^{j_G} &(G,\cX^+_G) &\rTo^{\rho_G} &(S^\prime,\cX_G^{\prime+}) \\
\dTo & &\dTo & &\dTo & &\dTo \\
S_{G_2} &\lTo^{[\rho^\prime_G]}_{\sim} &S_{G_1} &\rTo^{[j_G]} &S_G &\rTo^{[\rho_G]} &S^\prime_G
\end{diagram}.
\]
Therefore $\cX^+_{G_2}$ can be seen as a subset of $\cX^+_G$, and hence of $\H_g^+$. Since $\bar{Y\cap\Sigma_a^\prime}=Y$, we have $\bar{Y^\prime\cap[\rho](\Sigma_a^\prime)}=Y^\prime$. But then by the identification of $\cX_2^+$ and $\cX^{\prime+}$, we get that \textbf{in $S_2$, the subset of torsion points over $a^\prime$, where $A_{a^\prime}$ is isogenous to $A_a$, is Zariski dense in $Y_2$}.

For any torsion point $t$ over $a^\prime$ such that $A_{a^\prime}$ is isogenous to $A_a$, take a representative $\tilde{t}\in \unif_2^{-1}(t)$ in the fundamental set $\cF$ as in \cite[Section 10.1]{GaoTowards-the-And} (this fundamental set is similar to the one defined in Theorem~\ref{UniversalFamilyTheorem}.(3)). Denote by $V_2:=\cR_u(P_2)$, which is a $\Q$-vector group. Then $\tilde{t}=(\tilde{t}_V,\tilde{t}_G)\in V_2(\Q)\times(\H_g^+\cap M_{2g}(\bar{\Q}))$ and hence we can define its height. By choice of $\cF$, $H(\tilde{t}_V)$ is bounded by $N(t)$, the order of $t$ as a torsion point of $A_{a^\prime}$. But up to constants depending only on $a$ (or more explicitely, only on $H(\tilde{a})$), $H(\tilde{t}_G)$ is polynomially bounded from above by the minimal degree of the isogenies $A_{a^\prime}\rightarrow A_a$. This follows from \cite[Proposition 4.1, Section 4.2]{OrrFamilies-of-abe}. But the minimal degree of the isogenies $A_{a^\prime}\rightarrow A_a$ is polynomially bounded from above by the Galois orbit of $a^\prime$. This follows from \cite[Theorem 5.1]{OrrFamilies-of-abe}. Hence by \cite[Proposition 13.3]{GaoTowards-the-And},
\[
|\gal(\bar{\Q}/k)t|\gg_{g,\tilde{a}}H(\tilde{t})^{\mu(g,\tilde{a})}
\]
for some $\mu(g,\tilde{a})>0$. Hence for $H(\tilde{t})\gg0$, Pila-Wilkie \cite[3.2]{PilaO-minimality-an} implies that $\exists\sigma\in\gal(\bar{\Q}/k)$ such that $\tilde{\sigma(t)}$ is contained in a connected semi-algebraic subset of $\unif_2^{-1}(Y_2)\cap\cF$ of positive dimension. Now the Ax-Lindemann theorem (Theorem~\ref{Ax-Lindemann}) implies that $\sigma(t)$ is then contained in some weakly special subvariety $Z$ of $S_2$ such that $\dim Z>0$. Hence $\sigma^{-1}(Z)$ is weakly special containing a special point $t$, and therefore $\sigma^{-1}(Z)$ is special of positive dimension. To sum it up, the heights of the elements of
\begin{align*}
\{\tilde{t}\in \unif_2^{-1}(Y_2)\cap\cF\text{ special and }\unif_2(\tilde{t})\text{ is not contained in }\\
\text{a positive-dimensional special subvariety of }S_2\}
\end{align*}
is uniformly bounded from above. Therefore this set is finite by Northcott's theorem.
\end{enumerate}
\end{proof}

\section{Proof of the non-torsion case}\label{SectionProofOfTheNonTorsionCase}
We prove Theorem~\ref{MainTheoremNonTorsion} in this section. Let $Y$ be a curve in $\mathfrak{A}_g$, let $s\in\mathfrak{A}_g(\C)$ and let $\Sigma$ be the generalized Hecke orbit of $s$. For simplicity, we will denote by $(A,\lambda):=(\mathfrak{A}_{g,[\pi]s},\lambda_{[\pi]s})$ the polarized abelian variety attached to $[\pi](s)$ in this section.
Assume that $s$ is not a torsion point of $A$. Throughout this section, we assume that $Y$ is not contained in a fiber of $[\pi]\colon\mathfrak{A}_g\rightarrow\cA_g$ (otherwise this is a special case of the Mordell-Lang conjecture, which is proved in a series of works of Vojta, Faltings and Hindry).

We fix some notation here. Let $\cB$ be a symplectic basis of $H_1(A,\Z)$ with respect to the polarization $\lambda$. Let $\tilde{s}_G\in\H_g^+$ be the period matrix of $(A,\lambda)$ with respect to the basis $\cB$, then $\unif_G(\tilde{s}_G)=[\pi]s$. Now let $\tilde{s}=(\tilde{s}_V,\tilde{s}_G)\in V_{2g}(\R)\times\H^+_g\cong\cX^+_{2g,\mathrm{a}}$ be a point in $\pi^{-1}(\tilde{s}_G)\cap \unif^{-1}(s)$. In the whole section, we will fix $\cB$ to be the $\Q$-basis of $V_{2g}$ as in $\mathsection$\ref{SectionRationalRepresentationIsogeny}.

Denote by $k$ the definition field of $s$. Then $A$ is defined over the finitely generated field $k$.

\subsection{Complexity of points in a generalized Hecke orbit}\label{SubsectionComplexityNonTrosionCase}
Let $\unif\colon \cX^+_{2g,\mathrm{a}}\rightarrow\mathfrak{A}_g$ be the uniformization map and let $\cF$ be the fundamental set in $\cX^+_{2g,\mathrm{a}}$ defined in Theorem~\ref{UniversalFamilyTheorem}.(3). Let
\[
\tilde{Y}:=\unif^{-1}(Y)\cap\cF \text{ and } \tilde{\Sigma}:=\unif^{-1}(\Sigma)\cap\cF.
\]

Let $t\in\Sigma$. Let $f_t$ be as in Corollary \ref{GeneralizedHeckeIsogeny} (i.e. a polarized isogeny $(A,\lambda)\rightarrow(\mathfrak{A}_{g,[\pi]t},\lambda_{[\pi]t})$ of minimal degree). Define
\[
n_t:=\min\{n\in\N|~\exists\varphi\in\big(\End(A,\lambda)\big)\text{ such that }nt\in f_t\big(\varphi(s)+A(\C)_{\mathrm{tor}}\big)\}.
\]
The existence of such an $n_t$ is guaranteed by Corollary \ref{GeneralizedHeckeIsogeny}.
Furthermore, let $s_t:=\unif\left((\tilde{s}_V/n_t,\tilde{s}_G)\right)\in\mathfrak{A}_{g,[\pi]s}=A$. Then there exist by definition of $n_t$
\begin{itemize}
\item $\varphi_t\in\End\left((A,\lambda)\right)$;
\item $\delta_t$ a torsion point of $A$
\end{itemize}
such that
\begin{equation}\label{HeckeOribtFinalFormWeUse}
f_t\left(\varphi_t(s_t)+\delta_t\right)=t.
\end{equation}
The notation $n_t$, $f_t$, $\varphi_t$, $s_t$ and $\delta_t$ will be used throughout this section.

\begin{defn}\label{ComplexityNonTorsionCase}
Define the \textbf{complexity} of $t\in\Sigma$ to be
\[
\max\left(n_t,N(\delta_t)\right)
\]
where $N(\delta_t)$ is the order of $\delta_t$. In addition, define the \textbf{complexity} of any point of $\tilde{\Sigma}$ to be the complexity of its image in $\Sigma$.
\end{defn}
The fact that this complexity is a ``good enough" parameter will be proved in $\mathsection$\ref{GoodParameter}.

\subsection{Galois orbit}
In contrast to the torsion case, we deal with the Galois orbit at first for the non-torsion case. Keep the notation of the beginning of this section and $\mathsection$\ref{SubsectionComplexityNonTrosionCase}.

\begin{prop}\label{GaloisOrbitNonTorsionCase}
Let $t\in\Sigma$ be of complexity $n$, then
\[
[k(t):k]\geqslant c_3n^{c_4}
\]
where $c_3=c_3(A,\lambda,s)$ and $c_4=c_4(A,\lambda,s)$ are two positive constants.
\begin{proof} By \cite[Theorem~5.1]{OrrFamilies-of-abe} and \cite[Theorem~5.6]{MartinThesis}, there exist positive constants $c_5=c_5(A,\lambda)$ and $c_6=c_6(A,\lambda)$ such that
\begin{equation}\label{MinimumDegree}
\deg(f_t)\leqslant c_5[k(t):k]^{c_6}
\end{equation}

The abelian variety $A$ is defined over $k$. By the main result of \cite{MasserSmall-values-of} and the standard specialization argument introduced by Raynaud (see \cite[Section 5]{OrrFamilies-of-abe} or \cite[Section 7]{RaynaudCourbes-sur-une}), there exist two positive constants $c_9$ and $c_{10}$ depending only on $A$ and $k$ such that for any torsion point $q\in A$ of order $N(q)$, we have
\begin{equation}\label{MasserTorsionPoint}
[k(q):k]\geqslant c_9N(q)^{c_{10}}.
\end{equation}

$\framebox{\textit{Case i}}$ $N(\delta_t)^{c_{10}/2}\geqslant n_t^{2g^2+4g+1}$. By \cite[Proposition 1]{HindryAutour-dune-con} or \cite[Theorem 2.1.2]{McQuillanDivision-points} and the standard specialization argument introduced by Raynaud (see \cite[Section 5]{OrrFamilies-of-abe} or \cite[Section 7]{RaynaudCourbes-sur-une}), there exists a positive constant $c_{11}=c_{11}(A,s,k)$ such that
\[
\gal\left(k(\varphi_t(s_t),A[n_t])/k(A[n_t])\right)\leqslant c_{11}n_t^{2g}.
\]
Hence
\begin{equation}\label{NonTorsionPart}
[k(\varphi_t(s_t)):k]=|\gal\left(k(\varphi_t(s_t),A[n_t])/k(A[n_t])\right)|[k(A[n_t]):k]\leqslant c_{11}^\prime n_t^{2g^2+4g+1}
\end{equation}
for another positive constant $c_{11}^\prime$ depending only on $A$, $s$ and $k$.
Now by \eqref{NonTorsionPart}, \eqref{MasserTorsionPoint} and the assumption for this case,
\begin{equation}\label{NonTorsionBound}
[k(\varphi_t(s_t),\delta_t):k(\varphi_t(s_t))]\geqslant c_{12}\frac{N(\delta_t)^{c_{10}}}{n_t^{2g^2+4g+1}}\geqslant c_{12}N(\delta_t)^{c_{10}/2}
\end{equation}
for a positive constant $c_{12}=c_{12}(A,s,k)$.

Since $A$ is defined over the finitely generated field $k$, every element of $\mathrm{Aut}(\C/k)$ induces a homomorphism $A(\C)\rightarrow A(\C)$. It is not hard to prove the following claim:
\begin{claim} For any $\sigma_1$, $\sigma_2\in\mathrm{Aut}\left(\C/k(\varphi_t(s_t))\right)$, $\sigma_1(\varphi_t(s_t)+\delta_t)=\sigma_2(\varphi_t(s_t)+\delta_t)$ iff $\sigma_2^{-1}\sigma_1\in\mathrm{Aut}\left(\C/k(\varphi_t(s_t),\delta_t)\right)$.
\end{claim}
This claim implies $[k(\varphi_t(s_t)+\delta_t):k]\geqslant [k(\varphi_t(s_t),\delta_t):k(\varphi_t(s_t))]$. Hence by \eqref{NonTorsionBound},
\[
[k(\varphi_t(s_t)+\delta_t):k]\geqslant c_{12}N(\delta_t)^{c_{10}/2}.
\]
Since $t=f_t(\varphi_t(s_t)+\delta_t)$, we have therefore
\begin{equation}\label{StepBeforeRemovingFT}
[k(t):k]\geqslant c_{12}\frac{N(\delta_t)^{c_{10}/2}}{\deg(f_t)}.
\end{equation}
Now the conclusion for this case follows from \eqref{MinimumDegree}, \eqref{StepBeforeRemovingFT} and the definition of complexity (recall that $k$ is the definition field of $s$, and therefore depends only on $s$).

$\framebox{\textit{Case ii}}$ $N(\delta_t)^{c_{10}/2}\leqslant n_t^{2g^2+4g+1}$. Roughly speaking, this case follows from the Kummer theory \cite[Appendix 2]{HindryAutour-dune-con}. Here are the details of the proof:

Let $\Delta:=\End\left((A,\lambda)\right)s$ and let $\bar{\Delta}:=\End(A)s\subset A$. Then $\bar{\Delta}$ is a finitely generated subgroup of $A$. Let $k^\prime$ be the smallest field over which all points of $\bar{\Delta}$ are defined, then $k^\prime$ depends only on $A$ and $s$. Then $\bar{\Delta}\subset A(k^\prime)$. Let $\Delta^\prime:=\Q\Delta\cap A(k^\prime)$ and let $\bar{\Delta}^\prime:=\Q\bar{\Delta}\cap A(k^\prime)$. Then $\bar{\Delta}^\prime$ contains $\bar{\Delta}$. By the Lang-N\'{e}ron theorem, the group $A(k^\prime)$ is finitely generated (because $k^\prime$ is finitely generated over $\Q$). Therefore $\bar{\Delta}^\prime$ is finitely generated and $\mathrm{rank}\bar{\Delta}^\prime=\mathrm{rank}\bar{\Delta}$. Hence $[\bar{\Delta}^\prime:\bar{\Delta}]$ is a finite number depending only on $k^\prime$, and hence only on $A$ and $s$. On the other hand, $\Delta\subset\bar{\Delta}\cap\Delta^\prime\subset\Delta+A(k^\prime)_{\mathrm{tor}}$. So $[\bar{\Delta}\cap\Delta^\prime:\Delta]$ is a finite number depending only on $k^\prime$, and hence only on $A$ and $s$. Therefore by
\[
[\Delta^\prime:\Delta]=[\Delta^\prime:\bar{\Delta}\cap\Delta^\prime][\bar{\Delta}\cap\Delta^\prime:\Delta]\leqslant[\bar{\Delta}^\prime:\bar{\Delta}][\bar{\Delta}\cap\Delta^\prime:\Delta],
\]
there exists $c_{13}>0$ depending only on $A$ and $s$ such that $[\Delta^\prime:\Delta]=c_{13}$.

For each $t\in\Sigma$, define another number $n_t^\prime:=\min\{n\in\N|~nt\in f_t\big(A(k^\prime)+A(\C)_{\mathrm{tor}}\big)\}$. Let $s^\prime\in A(k^\prime)$ be such that $n_t^\prime t=f_t(s^\prime+A(\C)_{\mathrm{tor}})$. Then because $t=f_t(\varphi_t(s_t)+\delta_t)$, we have
\[
s^\dagger:=s^\prime-n_t^\prime\varphi_t(s_t)\in A(\C)_{\mathrm{tor}}.
\]
But $n_t^\prime\varphi_t(s_t)+s^\dagger\in\Delta^\prime$, so
\begin{equation}\label{DeltaPrime}
n_t^\prime=\min\{n\in\N|~nt\in f_t(\Delta^\prime+A(\C)_{\mathrm{tor}})\}.
\end{equation}
However by definition,
\begin{equation}\label{Delta}
n_t=\min\{n\in\N|~nt\in f_t(\Delta+A(\C)_{\mathrm{tor}}).
\end{equation}
Compare \eqref{DeltaPrime} and \eqref{Delta}, we get
\begin{equation}\label{ntAndntPrime}
n_t/n_t^\prime\leqslant[\Delta^\prime:\Delta]\leqslant c_{13}.
\end{equation}

By \cite[Lemma 14]{HindryAutour-dune-con} or \cite[Corollary 2.1.5]{McQuillanDivision-points} and the standard specialization argument introduced by Raynaud (see \cite[Section 5]{OrrFamilies-of-abe} or \cite[Section 7]{RaynaudCourbes-sur-une}), there exists a positive constant $c_{14}=c_{14}(A,k^\prime)$ such that
\[
\gal\Big(k^\prime\big(\varphi_t(s_t),A[n_t^\prime N(\delta_t)]\big)/k^\prime\big(A[n_t^\prime N(\delta_t)]\big)\Big)\geqslant c_{14}n_t^\prime.
\]
But $t=f_t(\varphi_t(s_t)+\delta_t)$, so
\begin{equation}\label{StepBeforeRemovingFT2}
[k(t):k]\geqslant[k^\prime(t):k^\prime]\geqslant\frac{[k^\prime(\varphi_t(s_t)+\delta_t):k^\prime]}{\deg(f_t)}\geqslant \frac{c_{14}n_t^\prime}{\deg(f_t)}.
\end{equation}
Now the conclusion follows from \eqref{MinimumDegree}, \eqref{ntAndntPrime} and \eqref{StepBeforeRemovingFT2}.
\end{proof}
\end{prop}

\subsection{N\'{e}ron-Tate height in family}\label{GoodParameter}
Next we prove that the complexity defined in Definition~\ref{ComplexityNonTorsionCase} is a good parameter. More explicitly we dispose of the following proposition:
\begin{prop}\label{GoodParameterProposition}
Let $Y$, $s$ and $\Sigma$ be as in the beginning of this section. Let $t\in\Sigma$. Let $f_t$, $n_t$, $s_t$, $\varphi_t$ and $\delta_t$ be as in $\mathsection$\ref{SubsectionComplexityNonTrosionCase}. Then
\[
\deg(\varphi_t)\leqslant c_{7}n_t^{c_{8}}\text{\quad and \quad}\deg(f_t)\leqslant c_7^\prime n_t^{c_8^\prime}
\]
for some positive constants $c_7=c_7(g,Y,s)$, $c_7^\prime=c_7^\prime(g,Y,s)$ and $c_8=c_8(g,Y,s)$, $c_8^\prime=c_8^\prime(g,Y,s)$.
\end{prop}

We shall prove this proposition with the help of a well-chosen family of N\'{e}ron-Tate heights, i.e. the one related to the symmetric and relatively ample $\G_m$-torsor $\mathfrak{L}_g\rightarrow\mathfrak{A}_g$ with respect to $\mathfrak{A}_g\rightarrow\cA_g$ defined in Theorem~\ref{AmpleLineBundleOverTheUniversalFamily}. We shall use the Moriwaki height (see \cite{MoriwakiArithmetic-heig}), which is defined for points over finitely generated fields. Then we shall use a theorem of Silverman-Tate \cite[Theorem A]{SilvermanHeights-and-the}.

Pink explained in \cite[Chapter 8 and 9]{PinkThesis} that $\mathfrak{L}_g$ extends over $\bar{\Q}$ to a relative ample $\G_m$-torsor $\bar{\mathfrak{L}_g}\rightarrow\bar{\mathfrak{A}_g}$ over $\bar{\mathfrak{A}_g}\rightarrow\bar{\cA_g}$, where $\bar{\mathfrak{A}_g}$ (resp. $\bar{\cA_g}$) is a compactification of $\mathfrak{A}_g$ (resp. $\cA_g$).\footnote{For experts of mixed Shimura varieties, we are in the situation of \cite[9.2]{PinkThesis} since we are considering (following Pink's notation) $(P_{2g},\cX^+_{2g})\rightarrow (P_{2g,\mathrm{a}},\cX^+_{2g,\mathrm{a}})$, so this follows from \cite[6.25, 8.6, 8.13, 9.13, 9.16, 9.24, 12.4]{PinkThesis}.} By abuse of notation we denote also by $\mathfrak{L}_g$ the relative ample line bundle associated to the $\G_m$-torsor. Let $\cM$ be an ample line bundle over $\bar{\Q}$ over $\cA_g$ which extends over $\bar{\Q}$ to an ample line bundle $\bar{\cM}$ over $\bar{\cA_g}$. For $a\gg 0$, the line bundle $\mathfrak{L}:=\mathfrak{L}_g\otimes[\pi]^*\cM^{\otimes a}$ over $\mathfrak{A}_g$ is ample.

Let $t\in\Sigma$ be as in Proposition~\ref{GoodParameter}. Recall that $k$ is the definition field of $s$. Hence $t\in\mathfrak{A}_g(\bar{k})$. Let $d$ be the transcendence degree of $k$ and let $\bar{\B}=(\B;\bar{H_1},...,\bar{H_d})$ be a big polarization of $k$, namely, a collection of a normal projective arithmetic variety $\B$ whose function field is $k$ and nef smooth hermitian line bundles $\bar{H_1},...,\bar{H_d}$ on $\B$ satisfying the bigness condition of Moriwaki \cite[pp~103,~above~Theorem~A]{MoriwakiArithmetic-heig}. Consider the arithmetic Moriwaki height associated to $\bar{\B}$
\[
h_{\mathfrak{A}_g,\mathfrak{L}}^{\bar{\B}}\colon\mathfrak{A}_g(\bar{k})\rightarrow\R
\]
defined in \cite[pp~103]{MoriwakiArithmetic-heig}.

For any point $b\in\cA_g(k)$, $\mathfrak{L}_{g,b}$ is an ample line bundle over the abelian variety $\mathfrak{A}_{g,b}$ defined over $k$. Now consider the N\'{e}ron-Tate height $\hat{h}_{\mathfrak{L}_{g,b}}^{\bar{\B}}$ on $A_b$ as in \cite[$\mathsection$3.4]{MoriwakiArithmetic-heig}. For any point $P\in\mathfrak{A}_g(k)$, we shall denote by
\[
\hat{h}_{\mathfrak{L}_g}^{\bar{\B}}(P):=\hat{h}_{\mathfrak{L}_{g,[\pi]P}}^{\bar{\B}}(P).
\]

\begin{lemma}\label{HeightsInGeneralizedHeckeOrbitComparison}
Let $s_1$ and $s_2$ be two points of $\mathfrak{A}_g(k)$. Assume that there exists a polarized isogeny
\[
f\colon(\mathfrak{A}_{g,[\pi]s_1},\lambda_{[\pi]s_1})\rightarrow(\mathfrak{A}_{g,[\pi]s_2},\lambda_{[\pi]s_2})
\]
such that $s_1=f(s_2)$. Then $\hat{h}_{\mathfrak{L}_g}^{\bar{\B}}(s_2)=(\deg f)^{1/g}\hat{h}_{\mathfrak{L}_g}^{\bar{\B}}(s_1)$.
\begin{proof} By the moduli interpretation of $\mathfrak{L}_g$ (Theorem~\ref{AmpleLineBundleOverTheUniversalFamily}), $f^*\mathfrak{L}_{g,[\pi]s_2}=\mathfrak{L}_{g,[\pi]s_1}^{\otimes(\deg f)^{1/g}}$. So we have
\begin{align*}
\hat{h}_{\mathfrak{L}_g}^{\bar{\B}}(s_2) &=\hat{h}_{\mathfrak{L}_{g,[\pi]s_2}}^{\bar{\B}}(f(s_1)) \\
&=\hat{h}_{\mathfrak{L}_{g,[\pi]s_1}^{\otimes(\deg f)^{1/g}}}^{\bar{\B}}(s_1) \\
&=(\deg f)^{1/g}\hat{h}_{\mathfrak{L}_{g,[\pi]s_1}}^{\bar{\B}}(s_1) \\
&=(\deg f)^{1/g}\hat{h}_{\mathfrak{L}_g}^{\bar{\B}}(s_1).
\end{align*}
\end{proof}
\end{lemma}

Now we start proving Proposition~\ref{GoodParameterProposition}.
\begin{proof}[Proof of Proposition~\ref{GoodParameterProposition}]
Denote by $\epsilon\colon\cA_g\rightarrow\mathfrak{A}_g$ the zero section.

Following Silverman \cite[$\mathsection$2,~pp~200]{SilvermanHeights-and-the}, we define the \textit{canonical height} $\hat{h}_\mathfrak{L}^{\bar{B}}$ by 
\[
\hat{h}_{\mathfrak{L}}^{\bar{\B}}(P):=\lim_{n\rightarrow\infty}n^{-2}h_{\mathfrak{A}_g,\mathfrak{L}}^{\bar{\B}}(nP),\quad\forall P\in\mathfrak{A}_g(k).
\]
Then
\[
\hat{h}_{\mathfrak{L}}^{\bar{\B}}=\hat{h}_{\mathfrak{L}_g}^{\bar{\B}}.
\]

Apply \cite[Theorem~A]{SilvermanHeights-and-the}: there exist constants $c_{15}=c_{15}(g)>0$ and $c_{16}=c_{16}(g)$ such that
\begin{equation}\label{HeightEquation1}
|\hat{h}_{\mathfrak{L}_g}^{\bar{\B}}(t)-h_{\mathfrak{A}_g,\mathfrak{L}}^{\bar{\B}}(t)|<c_{15}h_{\cA_g,\epsilon^*\mathfrak{L}}^{\bar{\B}}([\pi]t)+c_{16}
\end{equation}
for any $t\in \mathfrak{A}_g(k)$. Remark that the original theorem of Silverman is a statement for points over global fields, but his proof easily extends to points over finitely generated fields for the Moriwaki height \cite{MoriwakiArithmetic-heig}.

We need the following lemma, which uses the fact that $Y$ is a curve in an essential way:
\begin{lemma}\label{FirstDevissageNonTorsionCase}
There exist two constants $c_{17}>0$ and $c_{18}$ depending only on $Y$ such that
\[
h_{\mathfrak{A}_g,\mathfrak{L}}^{\bar{\B}}(t)\leqslant c_{17}h_{\cA_g,\epsilon^*\mathfrak{L}}^{\bar{\B}}([\pi]t)+c_{18}
\]
\begin{proof}
The idea is due to Lin-Wang \cite[Proof of Proposition 2.1]{LinIsogeny-orbits-}.
The following notation will be used only in this proof: denote by $B=[\pi](Y)$ and $X=[\pi]^{-1}(B)$. By abuse of notation, we will not distinguish $[\pi]$ and $[\pi]|_X$. Remark that $X\rightarrow B$ is a non-isotrivial family of abelian varieties.

Let $Y^\prime$ be a smooth resolution of $Y\subset\mathfrak{A}_g$, then $X\times_BY^\prime\rightarrow Y^\prime$ is also a non-isotrivial family of abelian varieties of dimension $g$ and we write $\epsilon_{Y^\prime}\colon Y^\prime\rightarrow X\times_BY^\prime$ to be the zero-section. Let $f\colon Y^\prime\rightarrow\mathfrak{A}_g$ be the natural morphism. Consider the following commutative diagram

\[
\begin{diagram}
X\times_B\SEpbk Y^\prime & \rTo_{p_2}\upperarrowtwo{\epsilon_{Y^\prime}} & Y^\prime \\
\dTo_{p_1} & & \dTo_{[\pi]\circ f} \\
X & \rTo^{[\pi]} & B \\
\end{diagram}.
\]
Now let $t^\prime\in Y^\prime(k)$ be such that $f(t^\prime)=t$. Then up to bounded functions,
\begin{align*}
h_{\mathfrak{A}_g,\mathfrak{L}}^{\bar{\B}}(t) &=h_{X,\mathfrak{L}_g|_{X}}^{\bar{\B}}(t)  &h_{\cA_g,\epsilon^*\mathfrak{L}}^{\bar{\B}}([\pi]t) &=h_{B,\epsilon^*\mathfrak{L}|_{X}}^{\bar{\B}}([\pi]t) \\
&=h_{X,\mathfrak{L}|_{X}}^{\bar{\B}}(f(t^\prime)) & & =h_{B,\epsilon^*\mathfrak{L}|_{X}}^{\bar{\B}}(f\circ[\pi](t^\prime))\\
&=h_{Y^\prime,f^*\mathfrak{L}|_{X}}^{\bar{\B}}(t^\prime) & & =h_{Y^\prime,(f\circ[\pi])^*\epsilon^*\mathfrak{L}|_{X}}^{\bar{\B}}(t^\prime)\\
& & &=h_{Y^\prime,\epsilon_{Y^\prime}^*p_1^*\mathfrak{L}|_{X}}^{\bar{\B}}(t^\prime).
\end{align*}

Since $Y$ is a curve, the morphism $[\pi]\circ f\colon Y^\prime\rightarrow B$ is finite. Therefore $p_1^*\mathfrak{L}|_{X}$ is ample. So $\epsilon_{Y^\prime}^*p_1^*\mathfrak{L}|_{X}$ is ample. Hence there exist two constants $c_{17}>0$ and $c_{18}$ depending only on $Y^\prime$ (and hence only on $Y$) such that
\begin{equation}\label{HeightEquation3}
h_{Y^\prime,f*\mathfrak{L}|_{X}}^{\bar{\B}}(t^\prime)\leqslant c_{17}h_{Y^\prime,\epsilon_{Y^\prime}^*p_1^*\mathfrak{L}|_{X}}^{\bar{\B}}(t^\prime)+c_{18}
\end{equation}
for any $t^\prime\in Y^\prime$. Now the conclusion follows.
\end{proof}
\end{lemma}

Now for any $t\in Y\cap\Sigma$, by \eqref{HeckeOribtFinalFormWeUse} and Lemma~\ref{HeightsInGeneralizedHeckeOrbitComparison},
\begin{equation}\label{ContradictionImplication2}
\hat{h}_{\mathfrak{L}_g}^{\bar{\B}}(t)=\frac{\deg(f_t)^{1/g}\deg(\varphi_t)^{1/g}}{n_t^2}\hat{h}_{\mathfrak{L}_g}^{\bar{\B}}(s).
\end{equation}

But for any $t\in \Sigma$, we have the following result of Moriwaki (\cite[Proposition~3.2~and~Lemma~1.6.3]{MoriwakiThe-modular-hei}):
\begin{equation}\label{ContradictionImplication3}
|h_F^{\bar{\B}}(A_{[\pi]t})-h_F^{\bar{\B}}(A)|\leqslant c_{19}\log\deg(f_t)
\end{equation}
where $c_{19}$ depends only on $\B$, and hence $k$. Here $h_F$ is the Faltings' modular height defined by Moriwaki in \cite[Proposition~3.4(1)]{MoriwakiThe-modular-hei} (which he denotes by $h_{\mathrm{mod}}^{\bar{\B}}$). This is the generalization of the stable Faltings height for abelian varieties over $\bar{\Q}$.

Moreover Moriwaki proved (\cite[Proposition~4.1]{MoriwakiThe-modular-hei}) that there exists a positive constant $c_{20}$ and $c_{21}$ depnding only on $g$, $\cM$ and $\B$ such that
\begin{equation}\label{ContradictionImplication4}
|c_{20}h_F^{\bar{\B}}(A_{[\pi]t})-h_{\cA_g,\epsilon^*\mathfrak{L}}^{\bar{\B}}([\pi]t)|\leqslant c_{21}
\end{equation}
for any $t\in\mathfrak{A}_g(k)$.

Now \eqref{HeightEquation1}, Lemma~\ref{FirstDevissageNonTorsionCase}, \eqref{ContradictionImplication2}, \eqref{ContradictionImplication3} and \eqref{ContradictionImplication4} together imply
\[
\frac{\deg(\varphi_t)^{1/g}}{n_t^2}\deg(f_t)^{1/g}\hat{h}_{\mathfrak{L}_g}^{\bar{\B}}(s)\leqslant(c_{15}+c_{17})c_{20}\Big(c_{19}\log\deg(f_t)+h_F^{\bar{\B}}(A)\Big)
+(c_{15}+c_{17})c_{21}+c_{16}+c_{18}.
\]

Since $\deg(\varphi_t)\geqslant 1$, we get that $\deg(f_t)$ is polynomially bounded in $n_t$.

On the other hand, letting $\deg(f_t)\rightarrow\infty$, we see that there exist two positive constants $M_0$ and $c_{22}$ depending on nothing such that $\deg(\varphi_t)^{1/g}\leqslant c_{22}n_t^2$ for any $t\in Y\cap\Sigma$ with $\deg(f_t)>M_0$. But if $\deg(f_t)\leqslant M_0$, then $\deg(f_t)$ takes values in a finite set $\{1,...,M_0\}$. So $\deg(\varphi_t)$ is bounded polynomially in $n_t$ from above.
\end{proof}

\subsection{Application of the Pila-Wilkie theorem}
Keep the notation of the beginning of this section and $\mathsection$\ref{SubsectionComplexityNonTrosionCase}.

\begin{prop}\label{PilaWilkieNonTorsionCase}
Let $Y$ and $\tilde{s}$ be as in the beginning of this section. Let $\epsilon>0$. There exists a constant $C=C(Y,s,\epsilon)>0$ with the following property:

For every $n\geqslant 1$, there exist at most $Cn^\epsilon$ definable blocks $B_i\subset\tilde{Y}$ such that $\cup B_i$ contains all points of complexity $n$ of $\tilde{Y}\cap\tilde{\Sigma}$.
\begin{proof} The proof starts with the following lemma:
\begin{lemma}\label{NonTorsionCaseRationalImage}
There exist constants $C^\prime$ and $\kappa^\prime$ depending only on $g$ and $\tilde{s}$ such that

For any $\tilde{t}\in\tilde{Y}\cap\tilde{\Sigma}$ of complexity $n$, there exists a $(v,h)\in P_{2g}(\Q)^+$ such that $(v,h)\cdot\tilde{s}=\tilde{t}$ and $H\left((v,h)\right)\leqslant C^\prime n^{\kappa^\prime}$.
\begin{proof} Let $t:=\unif(\tilde{t})$. Then $t\in\Sigma$ and therefore we dispose of a relation as \eqref{HeckeOribtFinalFormWeUse}. Let $f^\prime_t:=f_t\circ\varphi_t$, then $f^\prime_t\colon(A,\lambda)\rightarrow(\mathfrak{A}_{g,[\pi]t},\lambda_{[\pi]t})$ is a polarized isogeny. Moreover, there exists a $\delta^\prime_t\in A(\bar{\Q})_{\mathrm{tor}}$ such that $N(\delta^\prime_t)\leqslant N(\delta_t)\deg(\varphi_t)$ and
\begin{equation}\label{HeckeOrbitFinalFormWeUse2}
t=f^\prime_t(s_t+\delta^\prime_t).
\end{equation}

\begin{claim}
There exists a symplectic basis $\cB^\prime$ for $H_1(\mathfrak{A}_{[\pi]t},\Z)$ with respect to the polarization $\lambda_{[\pi]t}$ such that the height of $\gamma_{f^\prime}\in\GSp_{2g}(\Q)^+$ (the matrix expression of $f^\prime_t$ in coordinates $\cB$ with respect to $\cB^\prime$) is polynomially bounded in $\deg(f^\prime_t)=\deg(\varphi_t)\deg(f_t)$ from above (see the beginning of this section for $\cB$). 
\end{claim}

This claim follows from \cite[Proposition 4.1]{OrrFamilies-of-abe}: remark that $f^\prime_t$ is a polarized isogeny instead of an arbitrary isogeny, hence the endomorphism $q\in\End(A)$ in \cite[4.3]{OrrFamilies-of-abe} equals $[\deg\varphi_t]^{1/g}$, and therefore the $u\in(\End A)^*$ in \cite[4.6]{OrrFamilies-of-abe} can be taken to be $1_A$.

Then $\unif_G(\gamma_{f^\prime}\cdot\tilde{s}_G)=[\pi]s$. Besides let $\tilde{\delta}^\prime_t=(\tilde{\delta}^\prime_{t,V},\tilde{s}_G)\in\cF$ be such that $\unif(\tilde{\delta}^\prime_t)=\delta^\prime_t$. Then $\tilde{\delta}^\prime_{t,V}\in V_{2g}(\Q)$ and, by \eqref{HeckeOrbitFinalFormWeUse2} and \eqref{RationalRepresentationPolarizedIsogeny},
\[
\unif\Big(\gamma_{f^\prime}\big(\frac{\tilde{s}_V}{n_t}+\tilde{\delta}^\prime_{t,V},\tilde{s}_G\big)\Big)=t.
\]
So there exists an element $\gamma=(\gamma_V,\gamma_G)\in\Gamma$ such that
\[
\gamma\gamma_{f^\prime}\big(\frac{\tilde{s}_V}{n_t}+\tilde{\delta}^\prime_{t,V},\tilde{s}_G\big)=\tilde{t},
\]
i.e.
\[
\tilde{t}=\Big(\gamma_V+\gamma_G\gamma_{f^\prime}\big(\frac{\tilde{s}_V}{n_t}+\tilde{\delta}^\prime_{t,V}\big),\gamma_G\gamma_{f^\prime}\tilde{s}_G\Big)=\big(\gamma_V+\gamma_G\gamma_{f^\prime}\tilde{\delta}^\prime_{t,V},\frac{\gamma_G\gamma_{f^\prime}}{n_t}\big)\cdot\tilde{s}.
\]
Denote by
\[
(v,h):=\big(\gamma_V+\gamma_G\gamma_{f^\prime}\tilde{\delta}^\prime_{t,V},\frac{\gamma_G\gamma_{f^\prime}}{n_t}\big),
\]
then $(v,h)$ is an element of $P_{2g}(\Q)^+$ such that $(v,h)\tilde{s}=\tilde{t}$. Now we prove that $H\left((v,h)\right)$ is polynomially bounded in the complexity $n$ of $\tilde{t}$. To prove this, it suffices to prove that $n_t$, $H(\tilde{\delta}^\prime_{t,V})$, $H(\gamma_{f^\prime})$, $H(\gamma_G)$ and $H(\gamma_V)$ are all polynomially bounded in $n$.

The fact that $n_t$ is bounded by $n$ follows directly from the definition of complexity.

For $H(\tilde{\delta}^\prime_{t,V})$: because $\tilde{\delta}^\prime_t\in\cF\cong[0,N)^{2g}\times\cF_G$ (where $N$ is the level structure, and hence depend on nothing), we have $\tilde{\delta}^\prime_{t,V}\in[0,N)^{2g}$. Therefore $H(\tilde{\delta}^\prime_{t,V})$ is bounded up to a constant by the denominator of $\tilde{\delta}^\prime_{t,V}$, which equals $N(\delta^\prime_t)$. But $N(\delta^\prime_t)\leqslant\deg(\varphi_t)N(\delta_t)$, hence it suffices to bound both $\deg(\varphi_t)$ and $N(\delta_t)$ by $n$. Now $\deg(\varphi_t)$ is polynomially bounded in $n_t$, and hence by $n$, by Proposition~\ref{GoodParameterProposition}. By definition of complexity, $N(\delta_t)\leqslant n$.

For $H(\gamma_{f^\prime})$: by choice, $H(\gamma_{f^\prime})$ is polynomially bounded in $\deg(f_t)\deg(\varphi_t)$, which is polynomially bounded in $n_t$ by Proposition~\ref{GoodParameterProposition}. Hence $H(\gamma_{f^\prime})$ is polynomially bounded in $n$ by definition of complexity.

For $H(\gamma_G)$: remark $\gamma_G\gamma_{f^\prime}\tilde{s}_G=\pi(\tilde{t})\in\cF_G$. By \cite[Lemma 3.2]{PilaAbelianSurfaces}, $H(\gamma_G)$ is polynomially bounded in $|| \gamma_{f^\prime}\tilde{s}_G||$. Therefore $H(\gamma_G)$ is polynomially bounded, with constants depending on $||\tilde{s}_G||$, by $n$.

For $H(\gamma_V)$: remark $\gamma_V+\gamma_G\gamma_{f^\prime}\tilde{\delta}^\prime_{t,V}+\gamma_G\gamma_{f^\prime}\tilde{s}_V/n_t=\tilde{t}_V\in[0,N)^{2g}$ (where $N$ is the level structure, and hence depend on nothing). Therefore $H(\gamma_V)$ is polynomially bounded in $|| \gamma_G\gamma_{f^\prime}\tilde{\delta}_{t,V}+\gamma_G\gamma_{f^\prime}\tilde{s}_V/n_t||$. Therefore $H(\gamma_V)$ is polynomially bounded, with constants depending on $||\tilde{s}_V||$, by $n$.
\end{proof}
\end{lemma}

Let $\sigma\colon P_{2g}(\R)^+\rightarrow\cX^+_{2g,\mathrm{a}}$ be the map $(v,h)\mapsto (v,h)\cdot\tilde{s}$.

The set $R=\sigma^{-1}(\tilde{Y})=\sigma^{-1}(\unif^{-1}(Y)\cap\cF)$ is definable because $\sigma$ is semi-algebraic and $\unif|_{\cF}$ is definable. Hence we can apply the family version of the Pila-Wilkie theorem (\cite[3.6]{PilaO-minimality-an}) to the definable set $R$:  for every $\epsilon>0$, there are only finitely many definable block families $B^{(j)}(\epsilon)\subset R\times\R^m$ and a constant $C_1^\prime(R,\epsilon)$ such that for every $T\geqslant 1$, the rational points of $R$ of height at most $T$ are contained in the union of at most $C_1^\prime T^\epsilon$ definable blocks $B_i(T,\epsilon)$, taken (as fibers) from the families $B^{(j)}(\epsilon)$. Since $\sigma$ is semi-algebraic, the image under $\sigma$ of a definable block in $R$ is a finite union of definable blocks in $\tilde{Y}$. Furthermore the number of blocks in the image is uniformly bounded in each definable block family $B^{(j)}(\epsilon)$. Hence $\sigma(B_i(T,\epsilon))$ is the union of at most $C_2^\prime T^{\epsilon}$ blocks in $\tilde{Y}$, for some new constant $C_2^\prime(Y,\tilde{a},\epsilon)>0$.

By Lemma~\ref{NonTorsionCaseRationalImage}, for any point $\tilde{t}\in\tilde{Y}\cap\tilde{\Sigma}$ of complexity $n$, there exists a rational element $\gamma\in R$ such that $\sigma(\gamma)=\tilde{t}$ and $H(\gamma)\leqslant C^\prime n^{\kappa^\prime}$. By the discussion in the last paragraph, all such $\gamma$'s are contained in the union of at most $C_1^\prime(C^\prime n^{\kappa^\prime})^{\epsilon}$ definable blocks. Therefore all points of $\tilde{Y}\cap\tilde{\Sigma}$ of complexity $n$ are contained in the union of at most $C_1^\prime C_2^\prime C^{\prime\epsilon}n^{\kappa^\prime\epsilon}$ blocks in $\tilde{Y}$.
\end{proof}
\end{prop}

\subsection{End of proof of Theorem~\ref{MainTheoremNonTorsion}}\label{EndOfTheoremNonTorsionCase}
Now we are ready to finish the proof of Theorem~\ref{MainTheoremNonTorsion}. 

Let $\Sigma_1$ be the set of points $t\in Y\cap\Sigma$ such that there is a positive-dimensional block $B\subset\tilde{Y}$ with $t\in \unif(B)$. Let $Y_1$ be the Zariski closure of $\Sigma_1$. Let $k$ be a number field such that both $Y$ and $Y_1$ are defined over $k$.

Let $t$ be a point in $Y\cap\Sigma$ of complexity $n$. By Proposition~\ref{GaloisOrbitNonTorsionCase}, there exist positive constants $c_5$ and $c_6$ depending only on $(A,\lambda)$ and $s$ such that
\[
[k(t):k]\geqslant c_5n^{c_6}.
\]

All $\gal(\bar{k}/k)$-conjugates of $t$ are contained in $Y\cap \Sigma$ and have complexity $n$. By Proposition~\ref{PilaWilkieNonTorsionCase}, the preimages in $\cF$ of these points are contained in the union of $C(Y,s,c_6/2)n^{c_6/2}$ definable blocks, each of these blocks being contained in $\tilde{Y}$.

For $n$ large enough, $c_5n^{c_6}>Cn^{c_6/2}$. Hence for $n\gg0$, there exists a definable block $B\subset\tilde{Y}$ such that $\unif(B)$ contains at least two Galois conjugates of $t$, and therefore $\dim B>0$ since blocks are connected. So being in $\unif(B)$, those conjugates of $t$ are in $\Sigma_1$. But $Y_1$ is defined over $k$, so $t\in Y_1$.

In summary, all points of $Y\cap\Sigma$ of large enough complexity are in $\Sigma_1$. This excludes only finitely many points of $Y\cap\Sigma$. So $Y_1=Y$.

Let $\Sigma_2$ be the set of points $t\in Y\cap\Sigma$ such that there is a connected positive-dimensional semi-algebraic set $B^\prime\subset\tilde{Y}$ with $t\in \unif(B^\prime)$. Let $Y_2$ be the Zariski closure of $\Sigma_2$. By definition of blocks, $\Sigma_2=\Sigma_1$, and hence $Y_2=Y_1=Y$.

Now since $\dim(Y)=1$, the conclusion follows from Theorem~\ref{Ax-Lindemann}.

\section{Variants of the main conjecture}\label{SectionVariantsOfTheMainConjecture}
In the previous sections we have discussed the intersection of a subvariety of $\mathfrak{A}_g$ with the set of division points of the polarized isogeny orbit of a given point \eqref{ModuliInterpretationOfSigma}. The goal of this section is twofold: one is to replace the given point by a finitely generated subgroup of one fiber of $\mathfrak{A}_g\rightarrow\cA_g$ (remark that the fiber is an abelian variety), the other is to replace the polarized isogeny orbit by the isogeny orbit. In particular we will prove that although these changes to Conjecture~\ref{AndrePinkConjecture} a priori seem to generalize the conjecture, both can actually be implied by Conjecture~\ref{AndrePinkConjecture} itself.

In the rest of the section, fix a point $b\in\cA_g$, which corresponds to a polarized abelian variety $(A,\lambda):=(\mathfrak{A}_{g,b},\lambda_b)$. Let $\Lambda$ be any finitely generated subgroup of $A$.

\begin{thm}\label{AndrePinkFinitelyGeneratedSubgroup}
Let $Y$ be an irreducible subvariety of $\mathfrak{A}_g$. Let $\Sigma_0$ be the set of division points of the polarized isogeny orbit of $\Lambda$, i.e.
\[
\Sigma_0=\{t\in\mathfrak{A}_g|~\exists n\in\N\text{ and a polarized isogeny }f\colon(A,\lambda)\rightarrow(\mathfrak{A}_{g,[\pi]t},\lambda_{[\pi]t})\text{ such that }nt\in f(\Lambda)\}.
\]
Assume that Conjecture~\ref{AndrePinkConjecture} holds for all $g$. If $\bar{Y\cap\Sigma_0}^{\Zar}=Y$, then $Y$ is weakly special.
\begin{proof} The proof is basically the same as Pink \cite[Theorem 5.4]{PinkA-Combination-o} (how Conjecture~\ref{AndrePinkConjecture} implies the Mordell-Lang conjecture).

Suppose $\rank\Lambda=r-1$. Let $V_{2g}^r$ be the direct sum of $r$ copies of $V_{2g}$ as a representation of $\GSp_{2g}$. Then the connected mixed Shimura variety associated with $V_{2g}^r\rtimes\GSp_{2g}$ is the $r$-fold fiber product of $\mathfrak{A}_g$ over $\cA_g$, and so its fiber over $b$ is $A^r$. Denote by
\[
\sigma\colon\mathfrak{A}_g\times_{\cA_g}...\times_{\cA_g}\mathfrak{A}_g\rightarrow\mathfrak{A}_g
\]
the summation map (remark that both varieties are abelian schemes over $\cA_g$).

Now the homomorphisms
\[
\begin{array}{cccc}
P_{2g,\mathrm{a}} &=V_{2g}\rtimes\GSp_{2g}&\hookrightarrow V_{2g}^r\rtimes\GSp_{2g}\hookrightarrow &V_{2gr}\rtimes\GSp_{2gr} \\
&(v,h)&\mapsto \left((v,...,v),h)\right)\mapsto &\left((v,...,v),(h,...,h)\right)
\end{array}
\]
induce Shimura immersions
\[
\begin{diagram}
\mathfrak{A}_g &\rTo &\mathfrak{A}_g\times_{\cA_g}...\times_{\cA_g}\mathfrak{A}_g &\rTo &\mathfrak{A}_{gr} \\
\dTo^{[\pi]} & &\dTo & &\dTo \\
\cA_g &\rTo^{=} &\cA_g &\rInto &\cA_{gr}
\end{diagram}
\]
For simplicity we shall not distinguish a point in $\mathfrak{A}_g$ (resp. $\cA_g$) and its image in $\mathfrak{A}_{gr}$ (resp. $\cA_{gr}$). Then $\mathfrak{A}_{gr,b}=A^r$.

Fix generators $a_1$,...,$a_{r-1}$ of $\Lambda$ and set $a_r:=-a_1-...-a_{r-1}$. Let $\Lambda^\prime$ be the division group of $\Lambda$, i.e. $\Lambda^\prime=\{s|~\exists n\in\N\text{ such that }ns\in\Lambda\}\subset A$. Then \cite[Lemma 5.3]{PinkA-Combination-o} asserts that
\begin{equation}\label{LambdaDivisionSet}
\Lambda^\prime=\Lambda^*_{a_1}+...+\Lambda^*_{a_r}=\sigma(\Lambda^*_{a_1}\times...\times\Lambda^*_{a_r})
\end{equation}
where (as Pink defined) $\Lambda^*_{a_i}:=\{s\in A|~\exists m,n\in\Z\setminus\{0\}\text{ such that }ns=ma_i\}$.

Now consider
\[
\Lambda^\dagger:=\sigma^{-1}(Y)\cap\{f^r(\Lambda^*_{a_1}\times...\times\Lambda^*_{a_r})|~f\colon(A,\lambda)\rightarrow(\mathfrak{A}_{g,b^\prime},\lambda_{b^\prime})\text{ a polarized isogeny}\}.
\]
We have
\begin{align*}
\sigma(\Lambda^\dagger) &=Y\cap\sigma(\{f^r(\Lambda^*_{a_1}\times...\times\Lambda^*_{a_r})|~f\colon(A,\lambda)\rightarrow(\mathfrak{A}_{g,b^\prime},\lambda_{b^\prime})\text{ a polarized isogeny}\}) \\
&=Y\cap\{f^r\left(\sigma(\Lambda^*_{a_1}\times...\times\Lambda^*_{a_r})\right)|~f\colon(A,\lambda)\rightarrow(\mathfrak{A}_{g,b^\prime},\lambda_{b^\prime})\text{ a polarized isogeny}\} \\
&=Y\cap\{f^r(\Lambda^\prime)|~f\colon(A,\lambda)\rightarrow(\mathfrak{A}_{g,b^\prime},\lambda_{b^\prime})\text{ a polarized isogeny}\}\qquad\eqref{LambdaDivisionSet}.
\end{align*}

Because $\bar{Y\cap\Sigma_0}^{\Zar}=Y$, $Y\cap\{f(\Lambda^\prime)|~f\colon(A,\lambda)\rightarrow(\mathfrak{A}_{g,b^\prime},\lambda_{b^\prime})\text{ a polarized isogeny}\}$ is Zariski dense in $Y$ (as subsets of $\mathfrak{A}_g$). Therefore $\sigma(\Lambda^\dagger)$ is Zariski dense in $Y$ (as subsets of $\mathfrak{A}_g\times_{\cA_g}...\times_{\cA_g}\mathfrak{A}_g$, and hence as subsets of $\mathfrak{A}_{gr}$). Let $Y^\dagger$ be the Zariski closure of $\Lambda^\dagger$ in $\mathfrak{A}_g\times_{\cA_g}...\times_{\cA_g}\mathfrak{A}_g$. Then $Y^\dagger$ is also a subvariety of $\mathfrak{A}_{gr}$. Since taking Zariski closures commutes with taking images under proper morphisms, we deduce that $\sigma(Y^\dagger)=Y$. So there exists an irreducible component $Y^\prime$ of $Y^\dagger$ such that $\sigma(Y^\prime)=Y$.

For any polarized isogeny $f\colon(A,\lambda)\rightarrow(\mathfrak{A}_{g,b^\prime},\lambda_{b^\prime})$, the generalized Hecke orbit of $(a_1,...,a_r)\in A^r$ as a point on $\mathfrak{A}_{gr}$ contains $f^r(\Lambda^*_{a_1}\times...\times\Lambda^*_{a_r})$ by Corollary \ref{GeneralizedHeckeIsogenyStep1}. Therefore the intersection of $Y^\prime$ with generalized Hecke orbit of $(a_1,...,a_r)$ in $\mathfrak{A}_{gr}$ is Zariski dense in $Y^\prime$. Hence Conjecture~\ref{AndrePinkConjecture} for $\mathfrak{A}_{gr}$ implies that $Y^\prime$ is weakly special. Therefore $Y=\sigma(Y^\prime)$ is also weakly special by the geometric interpretation of weakly special subvarieties of $\mathfrak{A}_g$ and of $\mathfrak{A}_{gr}$ (Proposition~\ref{DescriptionWeaklySpecial}).
\end{proof}
\end{thm}

\begin{cor}\label{AndrePinkFinitelyGeneratedSubgroupIsogenyOrbit}
Let $Y$ be an irreducible subvariety of $\mathfrak{A}_g$. Let $\Sigma_0^\prime$ be the set of division points of the isogeny orbit of $\Lambda$, i.e.
\[
\Sigma_0^\prime=\{t\in\mathfrak{A}_g|~\exists n\in\N\text{ and an isogeny }f\colon A\rightarrow\mathfrak{A}_{g,[\pi]t}\text{ such that }nt\in f(\Lambda)\}.
\]
Assume that Conjecture~\ref{AndrePinkConjecture} holds for all $g$. If $\bar{Y\cap\Sigma_0^\prime}^{\Zar}=Y$, then $Y$ is weakly special.
\begin{proof} Recall Zarhin's trick (see \cite[Proposition 4.4]{MartinThesis}): for any isogeny $f\colon A\rightarrow A^\prime$ between polarized abelian varieties, there exists $u\in\End(A^4)$ such that $f^4\circ u\colon A^4\rightarrow (A^\prime)^4$ is a polarized isogeny.

Now let $i\colon\mathfrak{A}_g\hookrightarrow\mathfrak{A}_{4g}$ be the natural embedding. Then $\Lambda_4:=\End(A^4)i(\Lambda)$ is a finitely generated subgroup of $A^4=\mathfrak{A}_{4g,i(b)}$ and hence
\[
\Sigma_0^\prime\subset\{t\in\mathfrak{A}_{4g}|~\exists n\in\N\text{ and a polarized isogeny }f\colon(A^4,\lambda^{\boxtimes4})\rightarrow(\mathfrak{A}_{4g,[\pi]t},\lambda_{[\pi]t})\text{ such that }nt\in f(\Lambda_4)\}.
\]
Now the conclusion follows from Theorem~\ref{AndrePinkFinitelyGeneratedSubgroup}.
\end{proof}
\end{cor}

\end{document}